\documentclass[a4paper, 10pt, twoside, notitlepage]{amsart}

\usepackage[utf8]{inputenc}
\usepackage{color}
\usepackage{amsmath} 
\usepackage{amssymb} 
\usepackage{amsthm}
\usepackage{geometry}
\usepackage{graphicx}
\usepackage{esint}
\usepackage{cancel}
\usepackage[colorlinks=true,linkcolor=blue]{hyperref}

\theoremstyle{plain}
\newtheorem{thm}{Theorem}
\newtheorem{prop}{Proposition}[section]
\newtheorem{lem}[prop]{Lemma}
\newtheorem{cor}[prop]{Corollary}

\newtheorem{rmk}[prop]{Remark}

\newtheorem{thm*}{Theorem}

\makeatletter
\expandafter\newcommand\csname thethm*default\endcsname{\thethm*}
\newcommand{\thmstarnum}[1]{\expandafter\gdef\csname thethm*\endcsname{#1*}}
\thmstarnum{\thethm}
\expandafter\g@addto@macro\csname endthm*\endcsname{\thmstarnum{\thethm}}
\makeatother

\newtheorem{thm**}{Theorem}

\makeatletter
\expandafter\newcommand\csname thethm**default\endcsname{\thethm**}
\newcommand{\thmstarstarnum}[1]{\expandafter\gdef\csname thethm**\endcsname{#1**}}
\thmstarstarnum{\thethm}
\expandafter\g@addto@macro\csname endthm**\endcsname{\thmstarstarnum{\thethm}}
\makeatother

\newcommand {\R} {\mathbb{R}} 
 \newcommand {\N} {\mathbb{N}}
 
\newcommand {\p} {\partial}

\newcommand {\D} {\Delta}

\newcommand {\supp} {\text{supp}}

\newcommand{\inp}[1]{\left\langle #1\right\rangle}
\newcommand{\ol}[1]{\overline{#1}}

\title{Boundary Reconstruction for the Anisotropic Fractional Calder\'on Problem}
\author{Xiaopeng Cheng} 
\address{Max-Planck Institute for Mathematics in the Sciences, Inselstraße 22, 04103 Leipzig, Germany}
\email{xiaopeng.cheng@mis.mpg.de}

\author{Angkana Rüland}
\address{Institute for Applied Mathematics and Hausdorff Center for Mathematics, University of Bonn, Endenicher Allee, 53115 Bonn, Germany}
\email{rueland@uni-bonn.de}

\begin{document}

\begin{abstract}
In this article, we provide a boundary reconstruction result for the anisotropic fractional Calder\'on problem and its associated degenerate elliptic extension into the upper half plane. More precisely, considering the setting from \cite{FGKU21}, we show that the metric on the measurement set can be reconstructed from the source-to-solution data. To this end, we rely on the approach by Brown \cite{B01} in the framework developed in \cite{NT01} (see also \cite{KY02}) after localizing the problem by considering it through an extension perspective.
\end{abstract}

\maketitle

\section{Introduction}
\label{sec:intro}

The purpose of this article is to complement the ``interior'' reconstruction result for the fractional anisotropic Calder\'on problem from \cite{FGKU21} by a matching ``boundary'' reconstruction result. In particular, this implies that in the setting of \cite{FGKU21} \emph{no} apriori knowledge of the metric on the measurement set needs to be assumed.

In order to observe this, let us begin by recalling the set-up of the anisotropic fractional Calder\'on problem from \cite{FGKU21}. Consider $(M,g)$ a closed, compact, connected, smooth Riemannian manifold. Let $(-\D_g)$ denote the Laplace-Beltrami operator on $(M,g)$ with an orthonormal basis of eigenfunctions $(\phi_k)_{k \in \N_0}$ and eigenvalues $(\lambda_k)_{k \in \N_0}$. We assume that the eigenvalues are sorted, i.e., allowing for repetitions we have $0= \lambda_0<\lambda_1 \leq \lambda_2 \leq \lambda_3 \leq \cdots$. We then define the fractional Laplacian spectrally as follows: For $s \in (0,1)$ and $u \in C^{\infty}(M)$
\begin{align*}
(-\D_g)^s u := \sum\limits_{ k\in \N} \lambda_k^{s} (u,\phi_k)_{L^2(M, dV_g)} \phi_k.
\end{align*}
As in \cite{FGKU21}, we consider the following inverse problem associated with the fractional Laplacian:
Let $O\subset M$ be an open, non-empty set, let $f \in C_c^{\infty}(O)$ and define 
\begin{align}
\label{eq:source_to_sol}
L_{s,O}: C_c^{\infty}(O) \rightarrow H^s(O), \ f \mapsto u^f|_{O},
\end{align}
where $u^f$ is a solution to 
\begin{align*}
(-\D_g)^{s} u^f = f \mbox{ in } M,
\end{align*}
with $(u^f ,1)_{L^2(M, dV_g)} = 0$ and thus $u^f = \sum\limits_{ k = 1}^{\infty} \lambda_k^{-s} (f,\phi_k)_{L^2(M, dV_g)} \phi_k$.
In the seminal work \cite{FGKU21} the question whether the operator $L_{s,O}$ uniquely determines the structure of $(M,g)$ was investigated and answered positively (up to the natural gauges), provided that the structure of $(O,g|_{O})$ was apriori known.

\begin{thm}[Theorem 1 in \cite{FGKU21}]
\label{thm:FGKU21}
Let $s\in (0,1)$, assume that $(M_1,g_1)$, $(M_2, g_2)$ are smooth, closed, connected Riemannian manifolds of dimension $n\geq 2$. Let $(O_1,g_1|_{O})\subset (M_1,g_1)$ and $(O_2,g_2|_{O})\subset (M_2,g_2)$ be open, non-empty sets such that $(O_1,g_1|_{O})=(O_2,g_2|_{O})=(O,g)$ for some known metric $g$ on $O$. Let $L_{s,O}^{1}, L_{s,O}^2$ denote the source-to-solution operators associated with the metrics $g_1,g_2$ as in \eqref{eq:source_to_sol} and assume that $L_{s,O}^{1} = L_{s,O}^2$. Then there exists a diffeomorphism $\Phi$ such that $\Phi^* g_2 = g_1$.
\end{thm}

This seminal result thus provides a uniqueness result for the anisotropic fractional Calder\'on problem, a question whose local counterpart has remained widely open up to date at $C^{\infty}$ regularity. For results on the classical problem in regularities of the analytic category or for suitable replacements we refer to \cite{LU89,LU01,LTU03,GB09}. The nonlocal geometric uniqueness result of Theorem \ref{thm:FGKU21}, by now, has been extended to various contexts, see, for instance, \cite{QU22,C23,FKU24}.  A key ingredient of the resolution of the problem in the nonlocal setting was a reduction to a hyperbolic problem based on the Kannai transmutation formula after recovering the heat kernel from the measurement data. This result was further revisited with an elliptic extension perspective, highlighting ``hidden analyticity'' in the extended direction, in \cite{R23}.   

In the present article, we return to the question of whether the apriori knowledge of the structure of $(O,g|_{O})$ is necessary for the uniqueness result from Theorem \ref{thm:FGKU21}. To this end, we consider a slightly modified measurement operator: We define 
\begin{align}
\label{eq:source_to_sol_new}
\bar{L}_{s,O}: C_c^{\infty}(O) \rightarrow H^s(O), \ f \mapsto u^f|_{O},
\end{align}
where $u^f$ is a solution to 
\begin{align*}
\sqrt{\det(g)}(-\D_g)^{s} u^f = f \mbox{ in } M,
\end{align*}
with $(u^f ,1)_{L^2(M, dV_g)} = 0$.
We emphasize that this operator is a natural variant of the measurement operator from \cite{FGKU21} since in \cite{FGKU21} it is assumed that $g|_{O}$ is known and since $f \in C_c^{\infty}(O)$. In particular, in the case in which $g|_{O}$ is assumed to be known, the two operators $L_{s,O}$ and $\bar{L}_{s,O}$ carry the same information. Hence, in what follows, with slight abuse of notation, we will also refer to the operator $\bar{L}_{s,O}$ as a source-to-solution operator.

As our first main result, we highlight that the knowledge of the map $\bar{L}_{s,O}$  indeed suffices to uniquely reconstruct the manifold $(M,g)$ \emph{without} the apriori knowledge of the full ``boundary data'' $(O,g|_{O})$. 

\begin{thm}[Source-to-solution measurements and boundary reconstruction]
\label{thm:bdry_reconstr_ext_nonloc_source}
Let $s\in (0,1)$, assume that $(M,g)$ is a closed, connected $ C^4 $-regular Riemannian manifold of dimension $n\geq 2$.  Let $O\subset M$ be a non-empty open set and let $\bar{L}_{s,O}$ denote the source-to-solution operator from \eqref{eq:source_to_sol_new} associated with the metric $g$.
Then, for each $\alpha \in \mathbb{S}^{n-1}$ and for each $x_0 \in O$ there exists a sequence $(N_{k,\alpha})_{k \in \N}\subset \R$ with $N_{k,\alpha} \rightarrow \infty$ as $k \rightarrow \infty$ and a sequence of smooth functions $(f_{N_{k,\alpha}})_{k \in \N}$ with $\int\limits_{M} f_{N_{k,\alpha}} d V_g = 0$ such that
\begin{align*}
\lim\limits_{k \rightarrow \infty} N_{k,\alpha}^{2s + \frac{n}{2}} \langle f_{N_{k,\alpha}}, \frac{1}{\sqrt{\det(g)}}{\bar{L}_{s,O}} (f_{N_{k,\alpha}}) \rangle_{H^{-s}(M,dV_g), H^{s}(M,dV_g)}  = \tilde{c} (\tilde{C}_{\alpha}(g)(x_0))^{-2s}.
\end{align*}
Here 
\begin{align}
\label{eq:zero_order_symbol}
\tilde{ C}_{\alpha}(g)(x_0):=\sqrt{\sum\limits_{j, \ell=1}^{n} \det(g(x_0))^{\frac{1}{2s}} g_{j \ell}^{-1}(x_0) \alpha_j \alpha_{\ell} },
\end{align}
and $\tilde{c} = c_s^{-1}\hat{c}_s^{-2}(c_1 + c_2) $ for $c_1 = \int\limits_{0}^{\infty} t K_s^2(t) dt$, $c_{2}= \int\limits_{0}^{\infty} t K_{1-s}^2 (t) dt$, $c_s= - \frac{2^{2s-1}\Gamma(s)}{\Gamma(1-s)} \neq 0$ and $\hat{c}_s = 2^{-s} \Gamma(1-s)$. The function $\Gamma(\cdot)$ denotes the Gamma-function and $ K_s(\cdot), K_{1-s}(\cdot) $ denote the modified Bessel functions of the second kind.
The functions $f_{N_{k,\alpha}}$ are  independent of $g$ and their supports are localized in $B_{1/N_{k,\alpha}}(x_0) \cap M$. 
\end{thm}

Let us comment on this result. It shows that there exists a family of boundary conditions which allows us to reconstruct $\det(g(x_0))^{\frac{1}{2s}} \alpha \cdot g^{-1}(x_0) \alpha$ for all $x_0 \in O$ and $\alpha \in \mathbb{S}^{n-1}$ and thus to recover the metric $g$ in $ O $. Indeed, the bilinear form in $\tilde{C}_{\alpha}(g)(x_0)$ allows us to uniquely recover it from these data. 

\begin{cor}
\label{cor:metric}
Let $s\in (0,1)$.
Let $(M_j,g_j)$ for $j \in \{1,2\}$ be closed, connected $C^4$-regular Riemannian manifolds of dimension $n\geq 2$. 
Let $(O_1,g_1)\subset (M_1,g_1)$ and $(O_2,g_2)\subset (M_2,g_2)$ be open, non-empty sets such that $O_1=O_2=:O$. Let $\bar{L}_{s,O}^{1}, \bar{L}_{s,O}^2$ denote the source-to-solution operators from \eqref{eq:source_to_sol_new} associated with the metrics $g_1,g_2$.
Assume that 
\begin{align*}
\bar{L}_{s,O}^1= \bar{ L}_{s,O}^2.
\end{align*}
 Then, $g_1|_O = g_2|_O$.
\end{cor}

We view Theorem \ref{thm:bdry_reconstr_ext_nonloc_source} as a ``boundary reconstruction result''. The interpretation of this result as a ``genuine'' boundary reconstruction result will be given through an extension perspective in the following sections. Moreover, as in \cite{B01} (or \cite{NT01}) on which we build, this result is constructive. Given the data $\bar{L}_{s,O}$ and $O$, we can thus recover $(O, g|_O)$ uniquely.
As a consequence of Theorem \ref{thm:bdry_reconstr_ext_nonloc_source}, it is possible to strengthen the result from \cite{FGKU21} even further and to avoid apriori knowledge on the metric, see Theorem \ref{thm:FGKU_nonlocal} below.

A similar result as in Theorem \ref{thm:bdry_reconstr_ext_nonloc_source} also holds for the setting with solution-to-source measurement data. More precisely, we consider the operator
\begin{align}
\label{eq:sol_to_source}
\tilde{L}_{s,O}: C_c^{\infty}(O) \to H^{-s}(O), \ u \mapsto f^u|_{O},
\end{align}
where $f^u$ is obtained as the inhomogeneity associated with the function $u$:
\begin{align*}
\sqrt{\det(g)}(-\D_g)^s u = f^u \mbox{ in } M.
\end{align*}
With this notation fixed, we obtain the following analogue of Theorem \ref{thm:bdry_reconstr_ext_nonloc_source}.

\begin{thm}[Solution-to-source measurements and boundary reconstruction]
\label{thm:bdry_reconstr_ext_nonloc_sol}
Let $s\in (0,1)$ and assume that $(M,g)$ is a closed, connected $C^2$-regular Riemannian manifold of dimension $n\geq 2$. Let $O\subset M$ be a non-empty open set and let $\tilde{L}_{s,O}$ denote the solution-to-source operator from \eqref{eq:sol_to_source} associated with the metric $g$. Let $\alpha \in \mathbb{S}^{n-1}$. For each $x_0 \in M$ there exists a sequence of smooth functions $(u_N)_{N\in\N}$ such that
\begin{align*}
\lim\limits_{N \rightarrow \infty} N^{-2s + \frac{n}{2}} \langle u_N, \frac{1}{\sqrt{\det(g)}} \tilde{L}_{s,O} (u_N) \rangle_{H^{-s}(M, dV_g),H^s(M, dV_g)}  = \tilde{c} (\tilde{C}_{\alpha}(g)(x_0))^{2s},
\end{align*}
where 
\begin{align}
\label{eq:zero_order_symbol1}
\tilde{C}_{\alpha}(g)(x_0) := \sqrt{\sum\limits_{j, \ell=1}^{n} \det(g(x_0))^{\frac{1}{2s}} g_{j \ell}^{-1}(x_0) \alpha_j \alpha_{\ell} },
\end{align}
and $\tilde{c} = c_s(c_1 + c_2)$ for  $c_1 = \int\limits_{0}^{\infty} t K_s^2(t) dt$, $c_{2}= \int\limits_{0}^{\infty} t K_{1-s}^2 (t) dt$ and $c_s= - \frac{2^{2s-1}\Gamma(s)}{\Gamma(1-s)} \neq 0$ where $\Gamma(\cdot)$ denotes the Gamma-function and $ K_s(\cdot), K_{1-s}(\cdot) $ denote the modified Bessel functions of the second kind.
The functions $u_N$ are  independent of $g$ and their supports are localized in $B_{1/N}(x_0)\cap M$. Moreover, the rescaled functions $(N^{\frac{n}{4}-s} u_N)_{N \in \N}$ satisfy a uniform $H^{s}(M)$ bound as $N \rightarrow \infty$. 
\end{thm}

Again, as a direct corollary one obtains the recovery of the metric $g$:

\begin{cor}
\label{cor:metric2}
Let $s\in (0,1)$.
Let $(M_j,g_j)$ for $j \in \{1,2\}$ be closed, connected $C^2$-regular Riemannian manifolds of dimension $n\geq 2$. 
Let $(O_1,g_1)\subset (M_1,g_1)$ and $(O_2,g_2)\subset (M_2,g_2)$ be non-empty open sets such that $O_1=O_2=:O$. Let $\tilde{L}_{s,O}^{1}, \tilde{L}_{s,O}^2$ denote the source-to-solution operators from \eqref{eq:sol_to_source} associated with the metrics $g_1,g_2$ and suppose that $ \tilde{L}_{s,O}^1 =  \tilde{L}_{s,O}^2$. Then, $g_1|_O = g_2|_O$.
\end{cor}

Both in Corollaries \ref{cor:metric} and \ref{cor:metric2} our measurements are weighted variants of the source-to-solution and the solution-to-source operators, respectively. This is a consequence of the fact that the natural measurements are weighted with the volume form (see Theorems  \ref{thm:bdry_reconstr_ext_nonloc_source} and \ref{thm:bdry_reconstr_ext_nonloc_sol}). In what follows below, by adopting an extension perspective, we will observe that this is consistent with the natural measurements in the classical Calder\'on problem.

\subsection{Main ideas}

Let us discuss the main ideas behind the derivation of Theorems \ref{thm:bdry_reconstr_ext_nonloc_source}, \ref{thm:bdry_reconstr_ext_nonloc_sol}. Here we build on the strategy from \cite{B01} in the version from \cite{NT01} and construct oscillatory solutions. We believe that this construction of approximate solutions and their asymptotic behaviour is of interest in their own right. See also \cite{KV84,KV85} and \cite{SU88} for some of the first applications of such oscillatory solutions in the boundary reconstruction question for the classical Calder\'on setting. These allow us to reconstruct the symbol of the operators $\bar{L}_{s,O}$ and $\tilde{L}_{s,O}$. While it would also be possible to pursue this fully from a pseudodifferential point of view (see, for instance, \cite{SU88, CM20}), we here use a reformulation of the nonlocal problem from above as a local problem via an extension method. This relies on the seminal results of \cite{CS07} for the constant coefficient and of \cite{ST10} for the variable coefficient fractional Laplacian. We further refer to \cite{CG11,BGS15,ATW18,CFSS24} for extensions into various contexts of these ideas.
The advantage of the reformulation of the problem \eqref{eq:source_to_sol} through a local equation is that all the well-developed and available local tools can be invoked in what follows below.

\subsubsection{Boundary reconstruction for a class of degenerate elliptic PDE}
\label{sec:ell_PDE_bdary}

Before returning to the geometric problem from \cite{FGKU21} and Theorems \ref{thm:bdry_reconstr_ext_nonloc_source} and \ref{thm:bdry_reconstr_ext_nonloc_sol}, let us first consider the case of degenerate elliptic equations in the Euclidean upper half space. This setting will serve as the foundation of our argument also in the geometric context by using an extension argument in the next section.

We consider the following degenerate elliptic problem: For $s\in (0,1)$, $\phi\in C_c^{\infty}(\R^n \times \{0\})$ let $\tilde{u} \in \dot{H}^{1}(\R^{n+1}_+, x_{n+1}^{1-2s})$ be a solution to
\begin{align}
\label{eq:main_one1_intro}
\begin{split}
\nabla \cdot x_{n+1}^{1-2s} \tilde{\gamma} \nabla \tilde{u} & = 0 \mbox{ in } \R^{n+1}_+,\\
\tilde{u} & = \phi \mbox{ on } \R^n \times \{0\}.
\end{split}
\end{align}
Here and in what follows we assume that
\begin{align*}
\tilde{\gamma}(x, x_{n+1}) = \begin{pmatrix} \gamma(x) & 0 \\ 0& 1 \end{pmatrix}
\end{align*}
is a smooth, bounded, matrix-valued function such that $\tilde{\gamma}$ is uniformly elliptic and symmetric.
In particular, $\gamma$ only depends on the tangential variables and is independent of the normal one.
We consider the following measurements encoded in the (generalized) Dirichlet-to-Neumann map
\begin{align}
\label{eq:DtN1}
\begin{split}
\Lambda_{\gamma}: H^{s}(\R^n \times \{0\}) & \rightarrow H^{-s}(\R^n \times \{0\}),\\
\phi& \mapsto \Lambda_{\gamma}(\phi):=  \lim\limits_{x_{n+1}\rightarrow 0} x_{n+1}^{1-2s} \p_{n+1} \tilde{u}(\cdot, x_{n+1}).
\end{split}
\end{align}

In this setting, we prove the following boundary reconstruction result. It can be viewed as a ``PDE version'' associated with Theorem \ref{thm:bdry_reconstr_ext_nonloc_sol}.

\thmstarstarnum{\ref{thm:bdry_reconstr_ext_nonloc_sol}}
\begin{thm**}[Whole space PDE setting]
\label{thm:Eucl}
Let  $\gamma\in C^2(\R^n;\R^{n\times n})$ and let $ \tilde{\gamma}$ be as above. Let $s\in (0,1)$ and let $\alpha \in \mathbb{S}^{n-1}$. For each $(x_0,0) \in \R^n \times \{0\}$ there exists a sequence of solutions $(\tilde{u}_N)_{N\in \N} \subset \dot{H}^{1}(\R^{n+1}_+, x_{n+1}^{1-2s})$ of the bulk equation in \eqref{eq:main_one1_intro} with smooth boundary data $(\phi_N)_{N \in \N}$ such that
\begin{align*}
\lim\limits_{N \rightarrow \infty} N^{-2s + \frac{n}{2}} \langle \Lambda_{\gamma}(\phi_N) ,\phi_N  \rangle_{H^{-s}(\R^n), H^{s}(\R^n)}   = (c_1+c_2) (C_{\alpha}(\gamma)(x_0))^{2s},
\end{align*}
where 
\begin{align*}
C_{\alpha}(\gamma)(x_0):=\sqrt{\sum\limits_{j, \ell=1}^{n}  \gamma_{j \ell}(x_0) \alpha_j \alpha_{\ell} },
\end{align*}
and $c_1 = \int\limits_{0}^{\infty} t K_s^2(t) dt$ and $c_{2}= \int\limits_{0}^{\infty} t K_{1-s}^2 (t) dt$. The functions $\phi_N$ are supported in $B_{1/N}((x_0,0)) \cap (\R^n \times \{0\})$ and they are independent of $\gamma$. Moreover, the functions $(N^{\frac{n}{4}-s}\phi_N )_{N\in\mathbb{N}} $ satisfy a uniform $ H^s(M) $ bound as $ N\to\infty $. 
\end{thm**}

As in the local case, a direct corollary of this uniqueness property is the stability of the boundary reconstruction:

\begin{cor}
\label{cor:stab_Eucl}
Let $s\in (0,1)$ and $j \in \{1,2\}$.
Let $\gamma_j \in C^2(\R^n, \R^{n\times n})$ with $0 \leq C_1 |\xi|^2 \leq \xi \cdot \gamma_j(x) \xi \leq C_2 |\xi|^2 $ for some $0< C_1 \leq C_2 < \infty$ and for all $\xi \in \R^n$ and $x \in \R^n$.
 Then there exists a constant $C = C(s,n,C_1, C_2)>0$ such that
\begin{align*}
\|\gamma_1 - \gamma_{2}\|_{L^{\infty}(\R^n)} \leq C \|\Lambda_{\gamma_1}-\Lambda_{\gamma_2}\|_{\mathcal{L}(H^{s}(\R^n),H^{-s}(\R^n))}.
\end{align*}
\end{cor}

A similar result as that of Theorem \ref{thm:Eucl} also holds if one considers Neumann-to-Dirichlet measurement data (see Theorem \ref{prop:NtD} in Section \ref{sec:NtD} below). Theorem \ref{thm:Eucl} provides the analogue of the boundary reconstruction results by Brown \cite{B01} and Nakamura-Tanuma \cite{NT01} in the case of our class of degenerate elliptic equations. We use these together with an extension technique (see \cite{CS07,ST10}) in order to return to a similar result for the fractional Laplace-Beltrami operator in the next section.

\begin{rmk}
\label{rmk:Eucl}
We remark that as the above results all build on extremely localized solutions and, thus, in addition to the natural energy estimates, mostly on local properties of elliptic equations, the boundary reconstruction results remain valid in the context of the unbounded geometries from \cite{CGRU23} and \cite{FGKRSU25}.
\end{rmk}

\subsubsection{Returning to the manifold setting}
\label{sec:mfd}
Let us next return to the geometric setting. To this end, we reformulate the operator \eqref{eq:source_to_sol} in terms of a local problem extended in an additional dimension in local coordinates. More precisely, let $f\in C_c^{\infty}(O)$ with $\int\limits_{M} f dV_g = 0$ and consider $\tilde{u}^f \in \dot{H}^1(M \times \R_+ , x_{n+1}^{1-2s})$ the unique solution to
\begin{align}
\label{eq:Neumann_Dirichlet_main}
\begin{split}
(\p_{n+1} x_{n+1}^{1-2s} \p_{n+1} + x_{n+1}^{1-2s} \D_g) \tilde{u}^f & = 0 \mbox{ on } M \times \R_+,\\
c_s \sqrt{\det(g)}\lim\limits_{x_{n+1} \rightarrow 0} x_{n+1}^{1-2s} \p_{n+1} \tilde{u}^f & = f \mbox{ on } M,
\end{split}
\end{align}
where $\Delta_g$ denotes the Laplace-Beltrami operator on $(M,g)$ and
\begin{align}
\label{eq:const_CS}
c_s := - \frac{2^{2s-1}\Gamma(s)}{\Gamma(1-s)} \neq 0.
\end{align}
Then, recalling the source-to-solution operator $\bar{L}_{s,O}$ from \eqref{eq:source_to_sol_new}, by the results from \cite{CS07} in the constant and from \cite{ST10} in the variable coefficient setting, it holds that
\begin{align*}
\bar{L}_{s,O} (f) = \tilde{u}^f|_{O}.
\end{align*}
We also refer to \cite{BGS15} and \cite{CFSS24} where this is formulated explicitly in the setting of manifolds.
With this observation, Theorem \ref{thm:bdry_reconstr_ext_nonloc_source} can be recast in the following geometric extension framework. 

\thmstarnum{\ref{thm:bdry_reconstr_ext_nonloc_source}}
\begin{thm*}[Neumann-to-Dirichlet data and boundary reconstruction]
\label{thm:bdry_reconstr_ext}
Let $s\in (0,1)$ and  assume that $(M,g)$ is a closed, connected $C^4$-regular Riemannian manifold of dimension $n\geq 2$. Let $O\subset M$ be a non-empty open set and let $\bar{L}_{s,O}$ denote the source-to-solution operator from \eqref{eq:source_to_sol_new} associated with the metric $g$.

Then, for each $\alpha \in \mathbb{S}^{n-1}$ and each $(x_0,0) \in M \times \{0\}$ there exist
\begin{itemize}
\item a sequence $(N_{k,\alpha})_{k \in \N} \subset \R$ with $N_{k,\alpha} \rightarrow \infty$ as $k \rightarrow \infty$,
\item a sequence of solutions $(\tilde{u}_{N_{k,\alpha}})_{k\in \N} \subset \dot{H}^{1}(M\times \R_+, x_{n+1}^{1-2s})$ of the bulk equation in \eqref{eq:Neumann_Dirichlet_main} with smooth generalized Neumann boundary data $(f_{N_{k,\alpha}})_{k \in \N} $ with $\int\limits_{M} f_{N_{k,\alpha}} dV_g = 0$
\end{itemize} 
  such that
\begin{align*}
\lim\limits_{k \rightarrow \infty} N_{k,\alpha}^{2s + \frac{n}{2}} \langle f_{N_{k,\alpha}} , \frac{1}{\sqrt{\det(g)}} \bar{ L}_{s,O}(f_{N_{k,\alpha}}) \rangle_{H^{-s}(M,dV_g), H^{s}(M,dV_g)}   = \tilde{c} (\tilde{C}_{\alpha}(g)(x_0))^{-2s},
\end{align*}
where $\tilde{C}_{\alpha}(g)(x_0)$ is the constant from \eqref{eq:zero_order_symbol} and $\tilde{c} = c_s^{-1}\hat{c}_s^{-2}(c_1+c_2)$ for  $c_1 = \int\limits_{0}^{\infty} t K_s^2(t) dt$, $c_{2}= \int\limits_{0}^{\infty} t K_{1-s}^2 (t) dt$, $c_s= - \frac{2^{2s-1}\Gamma(s)}{\Gamma(1-s)} \neq 0$ and $ \hat{c}_s=2^{-s}\Gamma(1-s) $, where $\Gamma(\cdot)$ denotes the Gamma-function and $ K_s(\cdot), K_{1-s}(\cdot) $ denote the modified Bessel functions.
The functions $f_{N_{k,\alpha}}$ are supported in $B_{1/N_{k,\alpha}}((x_0,0)) \cap (M \times \{0\})$ and the sequence $( f_{N_{k,\alpha}})_{k \in \N} $ is independent of $g$. 
\end{thm*}

Let us comment on the origins of the constant $\tilde{c}=c_s^{-1}\hat{c}_s^{-2}(c_1+c_2)$: In contrast to the Dirichlet-to-Neumann whole space setting, the constant carries the additional factor $c_s^{-1}\hat{c}_s^{-2}$. Here, the factor $ c_s^{-1} $ originates from \eqref{eq:Neumann_Dirichlet_main} and the extension result, and the constant $ \hat{c}_s^{-2} $ is a consequence of the renormalization of approximate solutions for the Neumann case.

A similar result holds in case that the Dirichlet-to-Neumann data are given on the extension level.
In this case one considers
\begin{align}
\label{eq:Dirichlet_Neumann_main}
\begin{split}
(\p_{n+1} x_{n+1}^{1-2s} \p_{n+1} + x_{n+1}^{1-2s} \D_g)  \tilde{u}^{\phi} & = 0 \mbox{ on } M \times \R_+,\\
 \tilde{u}^{\phi} & = \phi \mbox{ on } M \times \{0\}.  
\end{split}
\end{align}
Then, by the results from \cite{CS07} in the constant and from \cite{ST10} in the variable coefficient setting, it holds that
\begin{align*}
\tilde{L}_{s,O}: C_c^{\infty}(O ) \rightarrow H^{-s}(O), \
\tilde{L}_{s,O} (\phi)(x) = c_s \sqrt{\det(g)} \lim\limits_{x_{n+1} \rightarrow 0} x_{n+1}^{1-2s} \p_{n+1} \tilde{u}^{\phi}(x,x_{n+1})|_{O},
\end{align*}
where $\tilde{L}_{s,O}$ is the operator from \eqref{eq:sol_to_source} and $c_s \neq 0$ is the constant from \eqref{eq:const_CS}.
With this in hand, Theorem \ref{thm:bdry_reconstr_ext_nonloc_sol} can be reformulated as follows.

\thmstarnum{\ref{thm:bdry_reconstr_ext_nonloc_sol}}
\begin{thm*}[Dirichlet-to-Neumann data and boundary reconstruction]
\label{thm:bdry_reconstr_ext_NtD}
Let $s\in (0,1)$
and  assume that $(M,g)$ is a closed, connected $C^2$-regular Riemannian manifold of dimension $n\geq 2$. Let $O\subset M$ be a non-empty open set and let $\tilde{L}_{s,O}$ denote the solution-to-source operator from \eqref{eq:sol_to_source} associated with the metric $g$. For each $(x_0,0) \in M \times \{0\}$ there exists a sequence of solutions $(\tilde{u}_N)_{N \in \N} \subset \dot{H}^{1}(M \times \R_+, x_{n+1}^{1-2s})$ of the bulk equation in \eqref{eq:Dirichlet_Neumann_main} with smooth boundary data $(\phi_N)_{N \in \N} $ such that
\begin{align*}
\lim\limits_{N \rightarrow \infty} N^{-2s + \frac{n}{2}} \langle  \frac{1}{\sqrt{\det(g)}} \tilde{L}_{s,O}(\phi_N), {\phi_N} \rangle_{H^{-s}(M, d V_g), H^{s}(M, dV_g)} =   c_s(c_1 + c_2) (\tilde{C}_{\alpha}(g)(x_0))^{2s}.
\end{align*}
where $\tilde{C}_{\alpha}(g)(x_0)$ is the constant from \eqref{eq:zero_order_symbol} and $c_1 ,c_2, c_s, \hat{c}_s$ are the constants from Theorem \ref{thm:bdry_reconstr_ext_nonloc_sol}.
The functions $\phi_N$ are supported in $B_{1/N}((x_0,0)) \cap (M \times \{0\})$ and they are independent of $g$. The rescaled functions $(N^{-s + \frac{n}{4}} \phi_N)_{N \in \N}$ satisfy a uniform $H^{s}(M)$ bound as $N \rightarrow \infty$.
\end{thm*}

Having reformulated the reconstruction problem through the degenerate elliptic extensions from \eqref{eq:Neumann_Dirichlet_main} and \eqref{eq:Dirichlet_Neumann_main}, the problem becomes an essentially local problem for a degenerate elliptic operator. In this context, we compute the symbol of the respective operators using and suitably adapting the ideas from \cite{B01} and \cite{NT01} to deduce desired reconstruction results.

\subsection{Applications}

The above boundary reconstruction results entail various further consequences.

\subsubsection{The anisotropic fractional Calder\'on problem without apriori information in the measurement domain}
As a first, direct consequence of the outlined results, as already highlighted above, it follows that in the result from \cite{FGKU21} it is not necessary to assume that the full metric in the underlying domain is known. It is possible to identify it up to the natural gauge invariance. Hence, the theorem from \cite{FGKU21} turns into the following result:

\begin{thm}[Combining \cite{FGKU21} with the boundary reconstruction from above]
\label{thm:FGKU_nonlocal}
Let $s\in (0,1)$, assume that $(M_1,g_1)$, $(M_2, g_2)$ are smooth, closed, connected manifolds of dimension $n\geq 2$. Let $O \subset M_1\cap M_2$ be a non-empty open set. Let $\bar{L}_{s,O}^{1}, \bar{L}_{s,O}^2$ denote the source-to-solution operators associated with the metrics $g_1,g_2$ as in \eqref{eq:source_to_sol_new}. Assume that 

\begin{align*}
\bar{ L}_{s,O}^{1} = \bar{ L}_{s,O}^2.
\end{align*}
 Then there exists a diffeomorphism $\Phi$ such that $\Phi^* g_2 = g_1$.
\end{thm}

An analogous observation also holds in the solution-to-source setting from Proposition 1.1 in \cite{R23}. 
\begin{thm}[Combining Proposition 1.1.~in \cite{R23} with the boundary reconstruction from above]
\label{thm:FGKU_nonlocal1a}
Let $s\in (0,1)$, assume that $(M_1,g_1)$, $(M_2, g_2)$ are smooth, closed, connected manifolds of dimension $n\geq 2$. Let $O \subset M_1\cap M_2$ be a non-empty open set. Let $\tilde{L}_{s,O}^{1}, \tilde{L}_{s,O}^2$ denote the solution-to-source operators associated from \eqref{eq:sol_to_source} with the metrics $g_1,g_2$ and assume that $\tilde{ L}_{s,O}^{1} = \tilde{ L}_{s,O}^2$. Then there exists a diffeomorphism $\Phi$ such that $\Phi^* g_2 = g_1$.
\end{thm}

\subsubsection{Relating the local and nonlocal Calder\'on problems without apriori information on the metric on the measurement domain}
Similar consequences also immediately apply to the relation between the local and nonlocal source-to-solution problems from \cite[Theorem 2]{R23}. Here it was proved that given the knowledge of $(O,g|_{O})$ for some open set $O$, then the nonlocal source-to-solution data determine the local source-to-solution data. With the results from above, it is no longer necessary to assume that the apriori knowledge of $g|_{O}$ is given.

\begin{thm}[Combining Theorem 2 from \cite{R23} with the boundary reconstruction from above]
\label{thm:local_nonlocal1}
Let $s\in (0,1)$.
Let $(M,g)$ be a smooth, closed, connected Riemannian manifold of dimension $n\geq 2$. Let $O \subset M$ be an open set. 
Let $L_{1,O}$ denote the source-to-solution map
\begin{align*}
\tilde{H}^{-1}(O) \ni f \mapsto v^f|_{O} \in H^1(O), 
\end{align*}
where $v^f $ with $(v^f, 1)_{L^2(M,dV_g)}=0$ is a solution to $(-\D_g) v^f = f$ in $(M,g)$ with $(f,1)_{L^2(M,dV_g)}=0$. Further, let $\bar{L}_{s,O}$ denote the source-to-solution map from \eqref{eq:source_to_sol_new}. 
Then, the knowledge of $\bar{L}_{s,O}$ determines $g|_O$ and $L_{1,O}$. 
More precisely, we have the following results:
\begin{itemize}
\item[(i)] 
Let $f\in C_c^{\infty}(O)$ with $(f,1)_{L^2(M,dV_g)}=0$ and let $\tilde{u}^f \in \dot{H}^1(M \times \R_+, x_{n+1}^{1-2s})$ with the property that $(\tilde{u}^f(\cdot, x_{n+1}),1)_{L^2(M, dV_g)}=0$ for all $x_{n+1} \in \R_+$ denote the  extension of $u^f$ (as in \eqref{eq:Dirichlet_Neumann_main}) and assume that $u^f$ with $(u^f, 1)_{L^2(M,dV_g)}=0$ solves $\sqrt{\det(g)}(-\D_g)^s u^f = f$ in $M$. Then, for any $f\in C_c^{\infty}(O)$, the function
\begin{align*}
\int\limits_{0}^{\infty} t^{1-2s} \tilde{u}^f(x,t) dt
\end{align*}
can be reconstructed from the pair $(f,\bar{L}_{s,O}(f))$ .
\item[(ii)]
With the notation as in (i), the following density result holds:
\begin{align*}
&\overline{\left\{ \left( f, \int\limits_{0}^{\infty} t^{1-2s} \tilde{u}^f(x,t)dt|_{O} \right) \in \tilde{H}^{-1}(O)\times H^1(O): \  f\in C_{c,\diamond}^{\infty}(O)  \right\}}^{\tilde{H}^{-1}(O)\times H^1(O)} \\
&= \{(f,L_{1,O}f): \ f\in \tilde{H}^{-1}(O) \mbox{ with } (f,1)_{L^2(M,dV_g)}=0\}.
\end{align*}
Here we have used the notation $C_{c,\diamond}^{\infty}(O):= \{f\in C_c^{\infty}(O): (f,1)_{L^2(M, dV_g)}=0\}$.
\end{itemize}
\end{thm}

As a consequence of Theorem \ref{thm:local_nonlocal1}, also any uniqueness result for the classical, local Calder\'on problem transfers to the analogous nonlocal one without any apriori knowledge of the metrics restricted to the measurement set.

As pointed out in Remark \ref{rmk:Eucl} analogous variants of the above results also hold in the unbounded, asymptotically Euclidean setting from \cite{CGRU23}.

\subsection{Relation to the literature}
\label{sec:lit}

Let us put the above results into the context of the associated literature on the fractional Calder\'on problem. The fractional Calder\'on problem had been introduced in the seminal article \cite{GSU20} in which also a first partial data uniqueness result for a Schrödinger variant of it with constant coefficients in the principal part was proved. Subsequent work dealt with uniqueness and stability in critical function spaces \cite{RS20, RS18, BCR24}, uniqueness of lower order terms in the presence of known anisotropic background metrics \cite{GLX17}, single measurement reconstruction and stability results \cite{GRSU20, R21} as well as reconstruction by monotonicity methods \cite{HL20, HL20a}. Moreover, in \cite{C20} a generalized Liouville transform was introduced allowing one to pass from a generalized conductivity to a Schrödinger version of the problem. Further related equations and reconstruction questions for lower order contributions, for instance, have been studied in \cite{RS19,LLR20,L21,L23,CdHS22,BGU21,CLR20,CR21,CdHS24}. We further highlight the exterior reconstruction result from \cite{CRZ22} which transfers some of the ideas from \cite{KV84} into the isotropic exterior data setting. We emphasize that the literature on the fractional Calder\'on problem is extensive by now and that this list is not exhaustive and that many further references can be found in the literature cited in these articles.

A novel, more recent strand of research on the fractional Calder\'on problem deals with reconstructing its principal part. This was proved in the seminal work \cite{FGKU21}, see also \cite{F24} for the first approach initiating this analysis. The ideas have been extended to various geometric contexts \cite{QU22,C23} and have been revisited by means of an extension perspective in \cite{R23}. Moreover,  \cite{FKU24} considers both the unique recovery of  an anisotropic metric and a potential. We also refer to the work \cite{FL24} for related entanglement principles for nonlocal operators on Riemannian manifolds. We further point to the very recent work \cite{FGKRSU25} in which an anisotropic metric is reconstructed from exterior measurements. Related to these works, also connections between local and nonlocal, isotropic or anisotropic Calder\'on problems have been established recently \cite{GU21,CGRU23,LLU23,R23,L24,LNZ24} in which it is shown that local uniqueness results can be transferred to nonlocal ones. It is the geometric strand of ideas to which our present article relates, by seeking to complement the interior uniqueness results from the literature by ``boundary reconstruction properties'' in isotropic or anisotropic settings. Here we are strongly inspired by the works \cite{B01, NT01, KV84, KV84, SU88} which we adapt to our degenerate elliptic contexts.

\subsection{Outline of the article}
\label{sec:outline}

The remainder of the article is structured as follows: After a short preliminaries section, in Section \ref{sec:DtN} we discuss the proofs of Theorems \ref{thm:Eucl}, \ref{thm:bdry_reconstr_ext_nonloc_sol} and \ref{thm:bdry_reconstr_ext_NtD}. The key step here is a suitable adaptation of the ideas of Brown \cite{B01} and Nakamura-Tanuma \cite{NT01} to our degenerate elliptic setting. This proceeds by first computing an approximate solution to our equation in Section \ref{sec:DtN_approx_proof} and then deducing suitable error bounds in Section \ref{sec:error_bounds_anisotropic} for the remainder term. Combined, this yields the result of Theorem \ref{thm:Eucl},  and, in Section \ref{sec:mfd1}, with an extension perspective it also provides the proofs of Theorems \ref{thm:bdry_reconstr_ext_nonloc_sol} and \ref{thm:bdry_reconstr_ext_NtD}.

In Section \ref{sec:NtD} we then show that this procedure can also be carried out for the setting of source-to-solution or Neumann-to-Dirichlet measurements and provide the proofs for Theorems \ref{thm:bdry_reconstr_ext_nonloc_source} and \ref{thm:bdry_reconstr_ext}. 

Finally, in Section \ref{sec:consequences} we discuss some applications of the boundary reconstruction results. In particular, we present the proof of Theorem \ref{thm:FGKU_nonlocal}.

\section{Preliminaries}

In this section, we collect some of the notation which we will use in what follows below. Most of this is standard notation, we recall it for the convenience of the reader and to introduce the specific conventions which we will use below. We will always assume that $n\geq 2$.

\subsection{Sets and approximation}\label{subsec:Preli}
We begin by recalling our notational conventions for sets.
\begin{itemize}
\item We use the convention that $0$ is not contained in $\N$ and use the notation $\N_0:= \N \cup \{0\}$.
\item We use the notation $\R^{n+1}_+:=\{(x,x_{n+1}) \in \R^{n+1}: \ x_{n+1}>0\}$. When working with points or vectors in $\R^{n+1}_+$, we will usually refer to them as $(x,x_{n+1})$ with $x \in \R^n$ denoting the tangential and $x_{n+1}$ the normal component. In what follows below, we will often identify $\R^n \times \{0\}$ with $\R^n$. We denote the Lebesgue measure on $\R^{n+1}_+$ by $dx dx_{n+1}$. With slight abuse of notation, we also omit it in integrals if there is no danger of confusion in the integration variables.
\item In what follows below, we will always assume that $(M,g)$ is a closed, connected Riemannian manifold of at least $C^2$  regular in Section \ref{sec:DtN} and of at least $ C^4 $ regular in Section \ref{sec:NtD}. Considering the Laplace-Beltrami operator $-\D_g$ on $(M,g)$, we denote its eigenvalues and eigenfunctions by $0=\lambda_0<\lambda_1 \leq \dots$ and $(\phi_k)_{k \in \N_0} $. Moreover, we denote the volume form on $(M,g)$ by $d V_g$.
\end{itemize}

In estimating the quality of our approximation results, in what follows below, we will make use of asymptotic behaviour and also use the $O$ and $o$ notation. More preicsely, we will use the following conventions:
\begin{itemize}
\item We will write $w(t ) = O(v(t))$ as $t \rightarrow c$ for $c \in \R \cup \{\pm \infty\}$ iff $\left|\frac{w(t)}{v(t)} \right| \leq C<\infty$ for $t \rightarrow c$. Similarly, we say $w(t) = o(v(t))$ as $t \rightarrow c$ for  $c \in \R \cup \{\pm \infty\}$ iff $\frac{w(t)}{v(t)} \rightarrow 0$ for $t \rightarrow c$.
\item We will also use the notation $w(t) \sim v(t)$ as $t \rightarrow c$ for $c \in \R \cup \{\pm \infty\}$ iff $\frac{w(t)}{v(t)} \rightarrow 1$ for $t \rightarrow c$.
\end{itemize}

\subsection{Function spaces and well-posedness}
In what follows below, we will mainly rely on standard Sobolev spaces. In addition to this, we will also use Sobolev spaces which are naturally associated with the extension perspective which we adopt in this article.

\begin{itemize}
\item For $s \geq 0$ we define 
\begin{align*}
H^{s}(M):=\{v: M \rightarrow \R: \ \|v\|_{H^{s}(M, dV_g)}<\infty\}.
\end{align*}
Here, $\|v\|_{H^{s}(M, dV_g)}^2:=\sum\limits_{k =0}^{\infty} (1+\lambda_k)^{2s} |(v, \phi_k)_{L^2(M, dV_g)}|^2$ and $(\phi_k)_{k \in \N_0}$, $(\lambda_k)_{k \in \N_0}$ denote the eigenfunctions and eigenvalues of the Laplace-Beltrami operator $-\D_g$ on $(M,g)$. 
For $s<0$, we similarly set 
\begin{align*}
H^{s}(M):= \mbox{ closure of } C^{\infty}( M ) \mbox{ with respect to } \|v\|_{H^{s}(M, d V_g)},
\end{align*}
where $\|v\|_{H^{s}(M, d V_g)}^2:= \sum\limits_{k = 0}^{\infty} (1+\lambda_k)^{2s}|(v, \phi_k)_{L^2(M, dV_g)}|^2$.
\item We further consider the following local Sobolev spaces. For $O \subset M$ Lipschitz regular and for $s \in \R$ we define $ \tilde{H}^s(O) $ to be the completion of $ C^\infty_c(O) $ with respect to the norm $ \|\cdot\|_{H^s(M, dV_g)} $. 
\item 
Since our measurements are localized to an open subset $O \subset M$, we also work with the spaces $H^s(O):=\{v|_O: \ v \in H^s(M)\}$ endowed with the quotient norm $\|v\|_{H^s(O, dV_g)}:= \inf\{ \|u\|_{H^s(M, dV_g)}: \ u \in H^s(M), \ u|_O = v \}$.
\item We further recall the homogeneous Sobolev spaces $\dot{H}^{-s}(\R^n)$ and $\dot{H}^s(\R^n)$. For $ s\in (-1,1) $, we let 
\[\dot{H}^s(\R^n):=\text{closure of }C^\infty_c(\R^n) \text{ with respect to }\|\cdot\|_{\dot{H}^s(\R^n)}, \]
where $ \|f\|_{\dot{H}^s(\R^n)}^2:=\int_{\R^n}|\xi|^{2s}|\hat{f}(\xi)|^2d\xi $ and $ \hat{f}(\xi):= \int\limits_{\R^n} e^{-i x \cdot \xi} f(x) dx $ denotes the Fourier transform of $ f $.
\end{itemize}
We recall that for $s\in \R$ we have $(H^{s}(M))^{\ast} = H^{-s}(M)$ and that under the above regularity assumptions on $O$ it also holds that $(\tilde{H}^{s}(O))^{\ast} = H^{-s}(O)$ for all $s \in \R$. The duality pairing between $ H^{-s}(M) $ and $ H^s(M) $ extends the $ L^2(M) $ product of smooth functions and is given by 
\begin{align*}
\inp{u,v}_{H^{-s}(M, dV_g), H^s(M, dV_g)}:=\sum_{k= 0}^ \infty(u,\phi_k)_{L^2(M, dV_g)}\overline{(v,\phi_k)_{L^2(M, dV_g)}},\\
\mbox{ for any }u\in H^{-s}(M) \cap C^{\infty}(M),\:v\in H^s(M)\cap C^{\infty}(M). 
\end{align*}

Associated with the extension perspective, we consider the following weighted spaces.
\begin{itemize}
\item The natural energy space associated with the extension problem \eqref{eq:main_one1_intro} is given by
\begin{align*}
\dot{H}^1(\R^{n+1}_+,x_{n+1}^{1-2s}):= \mbox{ closure of $C_c^{\infty}(\overline{\R^{n+1}_+})$ with respect to } \|\cdot\|_{\dot{H}^1(\R^{n+1}_+, x_{n+1}^{1-2s})}, 
\end{align*}
where $\|\tilde{u}\|_{\dot{H}^1(\R^{n+1}_+, x_{n+1}^{1-2s})}:= \|x_{n+1}^{\frac{1-2s}{2}} \nabla \tilde{u} \|_{L^2(\R^{n+1}_+)}$ and 
\begin{align*}
C_c^{\infty}(\overline{\R^{n+1}_+}):=\{\tilde{u} \in C^{\infty}(\R^{n+1}) \mbox{ with } \supp(\tilde{u}) \subset \overline{\R^{n+1}_+} \mbox{ compact}\}. 
\end{align*}
In particular, functions in $C_c^{\infty}(\overline{\R^{n+1}_+})$ do not need to vanish on $\R^{n} \times \{0\}$. We recall that this space is a Hilbert space with the scalar product 
\begin{align*}
(\tilde{u},\tilde{v})_{\dot{H}^1(\R^{n+1}_+, x_{n+1}^{1-2s})} := \int\limits_{\R^{n+1}_+} x_{n+1}^{1-2s} \nabla \tilde{u} \cdot \nabla \tilde{v} dx dx_{n+1},
\end{align*}
 and that the following trace and Sobolev-trace inequalities (see \cite[Sections 2.1 and 2.2]{BCPS13}) are valid
 \begin{align}
 \label{eq:trace}
 \begin{split}
 &\|\tilde{u}(\cdot,0)\|_{\dot{H}^{s}(\R^n)} \leq C \|x_{n+1}^{\frac{1-2s}{2}} \nabla \tilde{u} \|_{L^2(\R^{n+1}_+)}, \\
 &\|\tilde{u}(\cdot,0)\|_{L^{\frac{2n}{n-2s}}(\R^n)} \leq C \|x_{n+1}^{\frac{1-2s}{2}} \nabla \tilde{u} \|_{L^2(\R^{n+1}_+)}.
 \end{split}
 \end{align}
The quantities on the left hand side are understood in a trace sense.
\item The space $\dot{H}^1_0(\R^{n+1}_+, x_{n+1}^{1-2s})$ is defined as the closure of $C_c^{\infty}(\R^{n+1}_+)$ with respect to the seminorm $\|\tilde{u}\|_{\dot{H}^1(\R^{n+1}_+, x_{n+1}^{1-2s})}$. Together with the bilinear form $(\cdot,\cdot)_{\dot{H}^1(\R^{n+1}_+, x_{n+1}^{1-2s})}$ it becomes a Hilbert space. By virtue of the trace estimates \eqref{eq:trace}, we have that $\tilde{u}(x,0) = 0$ in a trace sense for all $\tilde{u} \in \dot{H}^1_0(\R^{n+1}_+, x_{n+1}^{1-2s})$.
 \item Occasionally, we will also work with the associated inhomogeneous spaces
 \begin{align*}
H^1(\R^{n+1}_+,x_{n+1}^{1-2s}):= \left\{ \tilde{u}: \R^{n+1}_+ \rightarrow \R: \ \|\tilde{u}\|_{H^1(\R^{n+1}_+, x_{n+1}^{1-2s})} < \infty \right\}, 
\end{align*}
where $\|\tilde{u}\|_{H^1(\R^{n+1}_+, x_{n+1}^{1-2s})}^2 := \|x_{n+1}^{\frac{1-2s}{2}}  \tilde{u} \|_{L^2(\R^{n+1}_+)}^2 + \|x_{n+1}^{\frac{1-2s}{2}} \nabla \tilde{u} \|_{L^2(\R^{n+1}_+)}^2 $. Together with the scalar product
\begin{align*}
(\tilde{u},\tilde{v})_{H^1(\R^{n+1}_+, x_{n+1}^{1-2s})} := \int\limits_{\R^{n+1}_+} x_{n+1}^{1-2s} \nabla \tilde{u} \cdot \nabla \tilde{v} dx dx_{n+1} + \int\limits_{\R^{n+1}_+} x_{n+1}^{1-2s} \tilde{u} \tilde{v} dx dx_{n+1},
\end{align*}
this also becomes a Hilbert space.
\item In Section \ref{sec:NtD}, for $x_0 \in \R^n$, we will also use the space $C^{\infty}_{x_0,\diamond}(\R^n ):=\{f\in C_c^{\infty}(B_1(x_0)): \ \int\limits_{\R^n} f dx = 0\}$, to obtain well-definedness of the generalized Neumann-to-Dirichlet operator.
\end{itemize}
A minimization problem or the lemma of Lax-Milgram yields the well-posedness of the problem \eqref{eq:main_one1_intro}.

\begin{prop}[Well-posedness]
\label{prop:well-posed}
Let $f\in H^{s}(\R^n)$. Then there exists a unique weak solution $\tilde{u}^f \in \dot{H}^{1}(\R^{n+1}_+, x_{n+1}^{1-2s})$ of \eqref{eq:main_one1_intro}. Moreover, there exists a constant $C>0$ such that any weak solution satisfies the following energy estimates
\begin{align}
\label{eq:apriori_est_CS}
\|x_{n+1}^{\frac{1-2s}{2}} \nabla \tilde{u}^f\|_{L^2(\R^{n+1}_+)} \leq C \|f\|_{H^{s}(\R^n)}.
\end{align}
\end{prop}

\begin{proof}
We consider the minimization problem
\begin{align*}
\inf\left\{ \int\limits_{\R^{n+1}_+}  x_{n+1}^{1-2s} \tilde{\gamma} \nabla \tilde{v} \cdot \nabla \tilde{v} dx dx_{n+1}: \ \tilde{v}(\cdot,0) =f \mbox{ on } \R^n \right\}.
\end{align*}
The identity on the boundary holds in a trace sense for which we use the trace estimates from \eqref{eq:trace}.
The non-emptiness of the set of admissible competitors follows from the solvability of the constant coefficient problem from \cite{CS07}. The existence of a minimizer $\tilde{u}^f$ follows from the direct method. It, hence, remains to deduce the claimed estimate.
To this end, we recall that for $\gamma = Id$, the minimizer $\tilde{v}^f$ from \cite{CS07} satisfies the bound \eqref{eq:apriori_est_CS}. As a consequence, 
\begin{align*}
0 &\leq \|x_{n+1}^{\frac{1-2s}{2}} \nabla \tilde{u}^f\|_{L^2(\R^{n+1}_+)}^2
= \inf\left\{ \int\limits_{\R^{n+1}_+}  x_{n+1}^{1-2s} \tilde{\gamma} \nabla \tilde{v} \cdot \nabla \tilde{v} dx dx_{n+1}: \ \tilde{v}(\cdot,0) =f \mbox{ on } \R^n \right\} \\
&\leq C_{\gamma}\inf\left\{ \int\limits_{\R^{n+1}_+}  x_{n+1}^{1-2s}| \nabla \tilde{v}|^2 dx dx_{n+1}: \ \tilde{v}(\cdot,0) =f \mbox{ on } \R^n \right\} \\
&= C_{\gamma} \|x_{n+1}^{\frac{1-2s}{2}} \nabla \tilde{v}^f\|_{L^2(\R^{n+1}_+)}^2\leq C \|f\|_{H^{s}(\R^n)}^2,
\end{align*}
which then also implies the estimate \eqref{eq:apriori_est_CS}.
\end{proof}

In local coordinates we will make use of various Dirichlet-to-Neumann and Neumann-to-Dirichlet operators. Thus, for completeness, we recall that
the weak form of the Dirichlet-to-Neumann map given in \eqref{eq:DtN1} is defined in the following sense. For  $ f\in H^s(\R^n\times\{0\}) $,
\begin{equation}\label{eq:DtN2_def}
	\inp{\Lambda_\gamma(f),g}_{H^{-s}(\R^n),H^s(\R^n)}:=\int_{\R^{n+1}_+}x_{n+1}^{1-2s}\tilde{\gamma}(x)\nabla \tilde{u}\cdot \nabla v dx dx_{n+1},\quad \text{ for all }g\in H^s(\R^n\times\{0\}),	
\end{equation}
where $ \tilde{u}\in \dot{H}^1(\R^{n+1}_+,x_{n+1}^{1-2s}) $ is the solution of \eqref{eq:main_one1_intro} with boundary data $ f $,  $ v $ is any $  \dot{H}^1(\R^{n+1}_+,x_{n+1}^{1-2s})$ function such that $ v|_{\R^n\times\{0\}}=g $, and the left hand side \eqref{eq:DtN2_def} holds in the sense of  the duality pairing between $ H^{-s}(\R^n\times\{0\}) $ and $ H^s(\R^n\times\{0\}) $. This is well-defined by the remarks on well-posedness of the equation \eqref{eq:main_one1_intro} and trace estimates for $ \dot{H}^1(\R^{n+1}_+,x_{n+1}^{1-2s}) $ functions.

\section{The Solution-to-Source Problem -- Boundary Reconstruction for the Anisotropic Extension Problem with Generalized Dirichlet-to-Neumann Data}
\label{sec:DtN}

In this section, we present the proof of Theorem \ref{thm:Eucl}. In order to later treat the geometric case in parallel, we formulate the problem in slightly more generality than needed for the proof of Theorem \ref{thm:Eucl}. More precisely, in this section, we seek to construct approximate solutions to the following anisotropic equation
\begin{align}
\label{eq:main_one1D}
\begin{split}
\nabla \cdot x_{n+1}^{1-2s} \tilde{\gamma}(x) \nabla \tilde{u} & = 0 \mbox{ in } \R^{n+1}_+,\\
\tilde{u} & = \phi \mbox{ on } \R^n \times \{0\},
\end{split}
\end{align}
where
\begin{align*}
\tilde{\gamma}(x, x_{n+1}) = c(x) \begin{pmatrix} \gamma(x) & 0 \\ 0& 1 \end{pmatrix}
\end{align*}
is a smooth, bounded, matrix-valued function and where $ c \in C^{\infty}(\R^n) $ with 
\begin{equation}
\label{eq:est_c}
0<c_1 \leq c(x)\leq c_2 < \infty \mbox{ for } x \in \R^n. 
\end{equation}
We assume that $\tilde{\gamma}$ is uniformly elliptic and symmetric.
Moreover, $\gamma$ and $c$ only depend on the tangential variables and are independent of the normal one. For the proof of Theorem \ref{thm:Eucl} it suffices to consider the case $c(x)= 1$ while for the geometric setting from Theorem \ref{thm:bdry_reconstr_ext_nonloc_sol} we will choose $c(x)= \sqrt{\det(g)(x)}$.
 
We then consider the following measurements encoded in the (generalized) Dirichlet-to-Neumann map
\begin{align}
\label{eq:DtN1a}
\begin{split}
\Lambda_{\gamma}: & \ H^{s}(\R^n \times \{0\}) \rightarrow H^{-s}(\R^n \times \{0\}),\\
&\phi \mapsto \Lambda_{\gamma}(\phi):=c(x) \lim\limits_{x_{n+1}\rightarrow 0} x_{n+1}^{1-2s} \p_{n+1} \tilde{u}(x,0).
\end{split}
\end{align}
Our objective is to reconstruct $c^{\frac{1}{s}}\gamma: \R^n \rightarrow \R^{n\times n}$ from this information. More precisely, also in this slightly generalized setting involving additionally the function $c$, as in Theorem \ref{thm:Eucl}, we seek to prove that there exists a sequence $(\phi_N)_{N \in \N} \subset H^{s}(\R^{n} )$ of boundary data which are independent of $\tilde{\gamma}$ such that for any $\alpha \in \mathbb{S}^{n-1}$ we have
\begin{align}
\label{eq:aim_PDE}
\lim\limits_{N \rightarrow \infty} N^{-2s + \frac{n}{2}} \langle \Lambda_{\gamma}(\phi_N) , \phi_N  \rangle_{H^{-s}(\R^n), H^{s}(\R^n)}   = (c_1+c_2) (\bar{C}_{\alpha}(\gamma)(x_0))^{2s},
\end{align}
where 
\begin{align}
\label{eq:constant_weighted}
\bar{C}_{\alpha}(\gamma)(x_0):=\sqrt{\sum\limits_{j, \ell=1}^{n}  c(x_0)^{\frac{1}{s}} \gamma_{j \ell}(x_0)  \alpha_j \alpha_{\ell} }.
\end{align}

Setting $c(x) \equiv 1$, then allows us to prove Theorem \ref{thm:Eucl} which can be viewed as a ``PDE version'' of the geometric formulation of Theorem \ref{thm:bdry_reconstr_ext_nonloc_sol}.

As a first step towards this, we seek to construct an explicit approximate solution. We make the ansatz that this approximate solution is given by 
\begin{align}
\label{eq:sol_approx}
\tilde{u}_N(x, x_{n+1}) = e^{i N \alpha \cdot x} \sum\limits_{r=0}^{2k} N^{-\frac{r}{2}} v_r(x, x_{n+1}),
\end{align}
and all functions $v_r$ are compactly supported and satisfy suitable Dirichlet conditions on $x_{n+1} = 0$. We follow the strategy from \cite{NT01} to derive the desired expansion for which we expand our operator into different orders of $N$. Based on this ansatz, we seek to deduce an approximate solution satisfying the bulk equation in \eqref{eq:main_one1D} with error $o(N^{2-k})$ provided $ \gamma\in C^{2k}(\R^n,\R^{n\times n}) $ for some $ k\geq 1 $.

\subsection{Construction of approximate solutions -- preliminaries}
\label{sec:DtN_approx_prelim}

Before turning to the construction of approximate solutions, we collect some properties of modified Bessel functions which we will exploit in what follows below. We recall that the modified Bessel functions $I_s(t), K_s(t)$ (see \cite[Chapter 10.25]{NIST}) form a fundamental system to the modified Bessel equation
\begin{align*}
t^2 w''(t) + t w'(t) - (s^2 +t^2) w(t) = 0, \ t >0.
\end{align*}
We use the notation from \cite{NIST} for these functions. 

 \begin{lem}[Identities for modified Bessel functions]
\label{lem:Bessel}
Let $s\in (0,1)$ and let $I_s, K_s$ denote the modified Bessel functions of the second kind (see, e.g., \cite{NIST}) and let $t>0$.
\begin{itemize}
\item[(a)] We have (see \cite{NIST} formulas (10.27.2) and (10.27.3)) that
\begin{align*}
I_{-s}(t) = I_s(t) + \frac{2}{\pi} \sin(s \pi) K_s(t), \ K_{-s}(t) = K_s(t).
\end{align*}
\item[(b)] The derivatives of $K_s, I_s$ are given by
\begin{align*}
K_s'(t) = \frac{s}{t} K_s(t) - K_{s+1}(t), \ 
I'_s(t) = - \frac{s}{t} I_s(t) + I_{s-1}(t).
\end{align*}
See \cite{NIST}, (10.29.2).
As a consequence,
\begin{align}
\label{eq:weighted_der}
 \frac{d}{dt} (t^{s} K_{s}(t)) = t^{s}K_{s-1}(t).
\end{align}
\item[(c)] The following asymptotic behaviour holds as $t \rightarrow 0$
\begin{align*}
K_s(t) \sim 2^{s-1} \Gamma(s) \frac{1}{t^s},\quad I_s(t)=o(1),\quad I_{s-1}(t)\sim\frac{t^{s-1}}{2^{s-1}\Gamma(s)},
\end{align*}
see \cite{NIST}, (10.30.2).
Further, as $t \rightarrow 0$,
\begin{align*}
K_{s-1}(t)= K_{1-s}(t)\sim 2^{-s} \Gamma(1-s) \frac{1}{t^{1-s}}.
\end{align*}
Moreover, combining (b) and (c), we also infer that
\begin{align*}
 \lim\limits_{t \rightarrow 0} t^{1-2s} \frac{d}{dt}(t^s K_{s}(t)) = 2^{-s}\Gamma(1-s), \ t>0.
\end{align*} 
\item[(d)] 
The following asymptotic behaviour holds as $ t\to \infty $
\[K_s(t)\sim \sqrt{\frac{\pi}{2}}t^{-1/2}e^{-t},\quad I_s(t)=O\left(\frac{e^t}{t^{1/2}} \right). \]
See \cite{NIST}, (10.25.2).
\end{itemize}
\end{lem}

Here and below, we use the notation $A(t)\sim B(t)$ as $t \rightarrow c$ to denote that $\frac{A(t)}{B(t)} \rightarrow 1$ as $t \rightarrow c$, see Section \ref{subsec:Preli}. 

Next, we consider an ODE which will be of central relevance in constructing approximate solutions.
We consider both the case of homogeneous and of the inhomogeneous version of the ODEs. The ODEs under consideration will naturally emerge after a separation of variables ansatz, distinguishing the tangential and normal variables in the bulk PDE from \eqref{eq:main_one1D}.

Indeed, seeking to construct approximate solutions to the bulk equation in \eqref{eq:main_one1D}, we separate the tangential and normal variables. For the tangential part we make the ansatz $v(x)=e^{i N \alpha \cdot x}$, which solves the eigenvalue problem  $\Delta_x v = -N^2 v$. It thus remains to only consider an ODE in the normal direction. This is given by
\begin{align}
\label{eq:normal}
t u''(t) + (1-2s) u'(t) = N^2 t u(t) \mbox{ for } t> 0.
\end{align}
We thus begin by deducing some structure results for solutions of \eqref{eq:normal}.

\begin{lem}
\label{lem:normal}
Let $s\in (0,1)$ and let $N \in \N$.
\begin{itemize}
\item[(a)] \emph{Homogeneous ODE}:
Set
\begin{align*}
h_N(t) := (Nt)^s K_{s}(Nt),
\end{align*}
where $K_s(t)$ is the modified Bessel function of the second kind with parameter $s$. Then the function $h_N$ satisfies \eqref{eq:normal}. Moreover, it has the following asymptotic expansion as $t \rightarrow 0$ and $t \rightarrow \infty$
	\begin{equation}\label{eq:homoODEsol}
	\begin{split}
		h_N(t)&\sim 2^{s-1}\Gamma(s)\text{ as }t\to 0,\\
		h_N(t)&\sim \sqrt{\frac{\pi}{2}}N^{-\frac{1-2s}{2}}t^{-\frac{1-2s}{2}}e^{-Nt}\text{ as }t\to\infty,\\
		h'_N(t)&\sim -2^{-s}\Gamma(1-s)N^{2s}t^{2s-1} \text{ as }t\to0,\\
		h'_N(t)&\sim -\sqrt{\frac{\pi}{2}}N^{\frac{1+2s}{2}}e^{-Nt}t^{-\frac{1-2s}{2}}\text{ as }t\to\infty.
	\end{split}
\end{equation}
\item[(b)] \emph{Inhomogeneous ODE}: Let $a\geq -\frac{1}{2}$, $v:[0,\infty) \rightarrow \R$ be such that
\begin{align*}
v(t) = O(t^{-s}) \mbox{ as } t\rightarrow 0 \mbox{ and } v(t) = O(t^{a} e^{-t}) \mbox{ as } t \rightarrow \infty.
\end{align*}
Then there exists a solution $w(t)$ of the inhomogeneous ODE 
\begin{align}
\label{eq:normal1}
t^2 w''(t) + t w'(t) - (s^2 + t^2) w(t) = t^2 v(t),
\end{align}
such that for any $A>0$ the following asymptotic behaviour holds
\begin{equation}\label{eq:inhomoODEsol}
\begin{split}
&	w(t) = \frac{o(1)}{t^s} \mbox{ as } t\rightarrow 0, \ w(t) = O(t^{a+1} e^{-t}) \mbox{ as } t \rightarrow \infty,\\
&	\frac{d}{dt}((At)^s w(At)) = A^{2s}O(t^{2s-1}) \mbox{ as } t \rightarrow 0,\\
&	\frac{d}{dt}((At)^s w(At)) = A^{s+a+2}O(t^{s+a+1}e^{-At}) \mbox{ as } t \rightarrow \infty.
\end{split}	
\end{equation}
The implicit constants in the last two estimates are independent of $A$.
\end{itemize}
\end{lem}

\begin{proof}
Let us first consider the homogeneous ODE. 	If  $ u(t) $ is a solution of \eqref{eq:normal}, using the ansatz $ u(t)=t^sv(t) $, and setting $ w(t):=v(t/N) $, we find that $ w(t) $ solves the following ODE
\begin{equation}\label{eq:BesselODE}
	t^2w''(t)+tw'(t)-(s^2+t^2)w(t)=0,\quad t>0.
\end{equation}
A fundamental system for this ODE is given by $ w(t)=I_s(t) \text{ and } w(t)=K_s(t)$  (see \cite{NIST},  (10.25.1)). 
Therefore, to ensure exponential decay as $ t\to\infty $, we choose $ w(t)=K_s(t) $ and define 
\begin{equation}\label{eq:hN}
	h_N(t):= (Nt)^sK_s(Nt).
\end{equation}
We note that, by construction, $ h_N(t) $ is a solution of \eqref{eq:normal}. The first two claims in \eqref{eq:homoODEsol} immediately follow from Lemma \ref{lem:Bessel} (c)(d). To show the claims for $ h_N'(t) $, we invoke \eqref{eq:weighted_der}. As a consequence, the last two claims in \eqref{eq:homoODEsol} follow from the asymptotic behaviour of $ K_{s-1}(t) $ given in Lemma \ref{lem:Bessel}. This proves part (a) of the lemma. 

For part (b), to solve the inhomogeneous ODE \eqref{eq:normal1}, we note that the Wronskian of $ K_s(t) $ and $ I_s(t) $ is (see (10.28.2) in \cite{NIST})
\[ \mathcal{W}(K_s(t),I_s(t))=\frac{1}{t},\quad t>0. \]
Hence, by variation of parameters, a solution of \eqref{eq:normal1} is given by
\[w(t)=AK_s(t)+BI_s(t)-K_s(t)\int_{0}^{t}I_s(\tau)v(\tau)\tau d\tau+I_s(t)\int_{0}^{t}K_s(\tau)v(\tau)\tau d\tau. \]
We next seek to choose the constants $ A,B $ appropriately in order to guarantee that $ w(t) $  decays exponentially as $ t\to\infty $. 
To this end, we choose $ B=-\int_{0}^{\infty}K_s(\tau)v(\tau)\tau d\tau $ (see \eqref{eq:choiceB} below) and then we  set
$ A=0 $. Observe that $ |B|<\infty $ by Lemma \ref{lem:Bessel} (c)(d) and since $ s\in(0,1) $. Hence, $ w(t) $ becomes 
\[w(t)=-I_s(t)\int_{t}^{\infty}K_s(\tau)v(\tau)\tau d\tau-K_s(t)\int_{0}^{t}I_s(\tau)v(\tau)\tau d\tau. \] 

It remains to prove the claims on the asymptotic behaviour of this function. In order to derive the first claim in \eqref{eq:inhomoODEsol},  we exploit the asymptotic behaviour given by Lemma \ref{lem:Bessel} (c)(d) and the assumption on $ v(t) $. These observations yield that as $ t\to\infty $
\begin{align}\tag*{}
	I_s(t)\int_{t}^{\infty}K_s(\tau)v(
	\tau)\tau d\tau&=O\left(\frac{e^t}{t^{1/2}}\right)O\left(\int_{t}^{\infty}e^{-2
		\tau}\tau^{a+\frac{1}{2}}d\tau \right)
	\\\label{eq:choiceB}&=O\left(\frac{e^t}{t^{1/2}}\right)O(e^{-2t}t^{a+\frac{1}{2}})=O(t^ae^{-t}),\\\tag*{}
	K_s(t)\int_{0}^{t}I_s(\tau)v(\tau)\tau d\tau&=O\left( \frac{e^{-t}}{t^{1/2}} \right)O\left(\int_{0}^{t}\tau^{a+\frac{1}{2}}d\tau \right)=O\left( \frac{e^{-t}}{t^{1/2}} \right)O(t^{a+\frac{3}{2}})=O(t^{a+1}e^{-t}).
\end{align}
Consequently, $ w(t)=O(t^{a+1}e^{-t}) $ as $ t\to\infty $.
Similarly, using the asymptotic behaviour for the functions $ I_s,K_s,v$, as $ t\to 0 $, we infer that
\[\begin{split}
	I_s(t)\int_{t}^{\infty}K_s(\tau)v(
	\tau)\tau d\tau&=o(1),\quad K_s(t)\int_{0}^{t}I_s(\tau)v(\tau)\tau d\tau=\frac{o(1)}{t^s},
\end{split} \]
and, hence, $ w(t)=\frac{o(1)}{t^s} $. This concludes the proof of the first claim in \eqref{eq:inhomoODEsol}.

Next, we provide the proof of the final two claims in \eqref{eq:inhomoODEsol}. We begin by considering the case $A=1$, the general case will then be obtained by rescaling.
Using the formulas for $ K_s $ and $ I_s $ specified in Lemma \ref{lem:Bessel} (a)(b), we obtain
\[\begin{split}
	\frac{d}{dt}&\left(t^sw(t) \right)= t^s\left(\frac{s}{t}w(t)+w'(t) \right)         \\
	& = t^s\left( -\frac{s}{t}I_s(t)\int_{t}^{\infty}K_s(\tau)v(\tau)\tau d\tau-\frac{s}{t}K_s(t)\int_{0}^{t}I_s(\tau)v(\tau)\tau d\tau \right.\\
	&\left. \quad -I_s'(t)\int_{t}^{\infty}K_s(\tau)v(\tau)\tau d\tau+K_s'(t)\int_{0}^{t}I_s(\tau)v(\tau)\tau d\tau  \right)\\	
	& =t^{s}\left(-I_{s-1}(t)\int_{t}^{\infty}K_s(\tau)v(\tau)\tau d\tau+ K_{s+1}(t)\int_{0}^{t}I_s(\tau)v(\tau)\tau d\tau \right.\\	
	&\left.  \quad -2\frac{s}{t}K_s(t)\int_{0}^{t}I_s(\tau)v(\tau)\tau d\tau \right).
\end{split}  \]
Now we rewrite $ \frac{s}{t}K_s(t)=-\frac{K_{s-1}(t)}{2}+\frac{K_{s+1}(t)}{2} $ (by Lemma \ref{lem:Bessel} (b)) which leads to
\[\begin{split}
	\frac{d}{dt}\left(t^sw(t)\right)
	& = t^s\left(-I_{s-1}(t)\int_{t}^{\infty}K_s(\tau)v(\tau)\tau d\tau  
	+ K_{s-1}(t)\int_{0}^{t}I_s(\tau)v(\tau)\tau d\tau \right). 
\end{split}  \]
Then, for $A=1$, the last two claims in \eqref{eq:inhomoODEsol} are obtained from the corresponding asymptotic behaviour of the functions $ K_s(t)$, $K_{s-1}(t)$, $I_{s}(t)$, $I_{s-1}(t)$, $v(t)$. More precisely, as $ t\to 0 $, 
\[\frac{d}{dt}\left(t^sw(t) \right)= t^s\left( O(t^{s-1})+O(t^{s-1})o(1) \right)= O(t^{2s-1}),  \]
and, as $ t\to\infty $, 
\[\begin{split}
	\frac{d}{dt}\left(t^sw(t) \right)&= t^s\left(O(e^{-t}t^a)+O(e^{-t}t^{a+1})\right)
	=O\left(e^{-t}t^{a+1+s}\right).
\end{split} \]
Finally, the results for $A>0$ follow by rescaling: For instance, 
\begin{align*}
\frac{d}{dt}\left((At)^sw(At)\right)
& = A \frac{d}{dz}\left(z^sw(z)\right)\big|_{z=At} = A O((At)^{2s-1}) = A^{2s}O(t^{2s-1}) \mbox{ as } t \rightarrow 0,
\end{align*}
with the implicit constants being independent of $A$.
\end{proof}

\subsection{Construction of approximate solutions -- expansion and conclusion}
\label{sec:DtN_approx_proof}

In this section, building on the results from the previous section, we construct suitable approximate solutions of \eqref{eq:main_one1D} which are of the form \eqref{eq:sol_approx}. We view the asymptotics of these solutions as of independent interest.

To this end, let us first fix the boundary data $ (\phi_N)_{N \in \N} $ which we will use in the boundary reconstruction result of Theorem \ref{thm:Eucl}. Without loss of generality, fix $ (x_0,0)=(0,0) $ and let $ \eta\in C^\infty_c(\R^n) $ be such that for some $C>0$
\begin{equation}
\label{eq:eta}
0\leq \eta\leq C,\quad \int_{\R^n}\eta^2dx=1,\quad \text{supp}(\eta)\subset\{|x|<1 \}, 
\end{equation}
and let $ \eta_N(x):=\eta(\sqrt{N}x) $. 
We will often use the following scaling 
\begin{equation}\label{eq:scaling}
	z_i=\sqrt{N}x_i,\:\:i=1,\ldots,N,\quad z_{n+1}=Nx_{n+1},
\end{equation}
which will help us to extract the dependence on $ N $ of various quantities in the computations.
For $ \alpha\in \mathbb{S}^{n-1} $ an arbitrary but fixed unit vector, we define the boundary data to be 
\begin{equation}\label{eq:bdydata}
	\phi_N(x):=\bar{c}_se^{iN\alpha\cdot x}\eta_N(x)
\end{equation}
where $ \bar{c}_s=2^{s-1}\Gamma(s) $ is the constant from the asymptotic behaviour of the modified Bessel function, i.e., it is such that $ K_s(t)\sim \bar{c}_st^{-s} $ as $ t\to 0 $.
For notational convenience, we also set
\begin{equation}\label{eq:C0}
	C_0:=C_0(\alpha):=\sqrt{\sum_{j,l=1}^n\gamma_{jl}(0)\alpha_j\alpha_l} .	
\end{equation}

\begin{prop}[Approximate solutions]
\label{prop:DtN_approx}
Let $s\in (0,1)$, $\alpha \in \mathbb{S}^{n-1}$, $N \in \N$ and let $\gamma \in C^{2k}(\R^n, \R^{n\times n}),\:c\in C^{2k}(\R^n)$ for some $k\geq 1$. Then there exists an approximate solution $\tilde{u}_N$ of the form \eqref{eq:sol_approx} with the following properties:
\begin{itemize}
\item[(a)] The functions $v_r$ have the following form
\begin{align*}
v_0(x,x_{n+1}) &= \eta(z) (C_0 z_{n+1})^s K_s(C_0 z_{n+1}),\\
v_r(x,x_{n+1}) &= \sum\limits_{\ell=1}^{r}\sum_{\rho=1}^{\binom{r-1}{\ell-1}} P_r^{\ell,\rho}(z) (C_0 z_{n+1})^s w_r^{\ell,\rho}(C_0 z_{n+1}), \  r\in \{1,\dots,2k\},
\end{align*}

where $P_{r}^{\ell,\rho}$ are smooth functions with compact support in $\{z\in \R^n: |z|\leq 1\}$ and for $r\in\{1,\dots,2k\}$ and $ \rho\in\{1,\ldots,\binom{r-1}{\ell-1} \} $ it holds that
\begin{align*}
w_r^{\ell,\rho}(t) &= \frac{o(1)}{t^{s}} \mbox{ as } t \rightarrow 0,\\
w_r^{\ell,\rho}(t) &= O(e^{-t} t^{\frac{1}{2}+(\ell-1)}) \mbox{ as } t \rightarrow \infty,\\
\lim\limits_{t \rightarrow 0} t^{1-2s} A^{-2s}\frac{d}{dt}\left( (At)^s w_r^{\ell,\rho}(At) \right) &=  O(1) ,\\
\lim_{t\to\infty}e^{At}t^{-s-\frac{1}{2}-(\ell-1)}A^{-s-\frac{1}{2}-\ell}\frac{d}{dt}\left((At)^sw_r^{\ell,\rho}(At) \right)&=O(1),
\end{align*} 
for any constant $A>0$. In the last two limits the implicit constants are independent of $A>0$.
\item[(b)] It holds that
\begin{align*}
\tilde{u}_N(x,0)
=  \bar{c}_s e^{i N \alpha \cdot x} \eta_N(x) =: \phi_N(x),
\end{align*}
i.e., with $\phi_N$ as in \eqref{eq:bdydata}.
\item[(c)] The function $\tilde{u}_N$ is an approximate solution in the sense that
\begin{align*}
\nabla \cdot x_{n+1}^{1-2s} \tilde{\gamma} \nabla \tilde{u}_N =  o(N^{1-k+2s}) e^{i N \alpha \cdot x} z_{n+1}^{1-2s} \mbox{ as } N \rightarrow \infty.
\end{align*}
\end{itemize}
\end{prop}

In constructing these solutions, we deduce a hierarchy of equations satisfied by the functions $v_{r}$. We will reduce these equations to certain (inhomogeneous) ODEs, which we can solve by virtue of Lemma \ref{lem:normal}.

\begin{proof}
\emph{Step 1: Expansion of the bulk equation and derivation of a hierarchy of PDEs.}
In order to construct the approximate solution, following \cite{NT01}, we rewrite the bulk equation in \eqref{eq:main_one1D}:
\begin{align*}
\nabla \cdot x_{n+1}^{1-2s} \tilde{\gamma}(x) \nabla \tilde{u}_N(x,x_{n+1})
& = \sum\limits_{j,\ell=1}^{n} x_{n+1}^{1-2s} \tilde{\gamma}_{j \ell}(x) \p^2_{x_j x_{\ell}} \tilde{u}_N (x, x_{n+1}) \\
& \quad +c(x) \p_{x_{n+1}} x_{n+1}^{1-2s } \p_{x_{n+1}}\tilde{u}_N (x, x_{n+1})\\
& \quad + \sum\limits_{j,\ell=1}^{n} x_{n+1}^{1-2s} \left( \p_{x_{j}} \tilde{\gamma}_{j \ell}(x)  \right) \p_{x_{\ell}} \tilde{u}_N(x,x_{n+1}).
\end{align*}
Using the ansatz \eqref{eq:sol_approx} for $\tilde{u}_N$ then leads to
\begin{align*}
&\nabla \cdot x_{n+1}^{1-2s} \tilde{\gamma}(x) \nabla \tilde{u}_N(x,x_{n+1})
= e^{i N \alpha \cdot x}\left[- N^2 \sum\limits_{j,\ell=1}^{n} x_{n+1}^{1-2s} \tilde{\gamma}_{j \ell}(x) \alpha_j \alpha_{\ell} \right.\\
& \quad + i N \sum\limits_{j,\ell=1}^{n} x_{n+1}^{1-2s} \tilde{\gamma}_{j \ell}(x) \left( \alpha_j \p_{x_{\ell}} +  \alpha_{\ell} \p_{x_j} \right) \\ 
&  \quad + \sum\limits_{j,\ell=1}^{n} x_{n+1}^{1-2s} \tilde{\gamma}_{j \ell}(x) \p^2_{x_{\ell} x_{j}}
+ c(x)\p_{x_{n+1}} x_{n+1}^{1-2s } \p_{x_{n+1}} \\
& \quad \left. 
+ i N\sum\limits_{j,\ell=1}^{n} x_{n+1}^{1-2s} \left( \p_{x_{j}} \tilde{\gamma}_{ j \ell}(x)  \right) \alpha_j 
+ \sum\limits_{j,\ell=1}^{n} x_{n+1}^{1-2s} \left( \p_{x_{j}} \tilde{\gamma}_{ j\ell}(x)  \right) \p_{x_{\ell}} \right] \sum\limits_{r=0}^{2k}  N^{-\frac{r}{2}} v_r(x, x_{n+1}).
\end{align*}
	Recalling \eqref{eq:scaling}, we rescale this identity. In particular, $ \partial_{x_j}=\sqrt{N}\partial_{z_j} $ for $ j=1,\ldots,n $ and $ \partial_{x_{n+1}}=N\partial_{z_{n+1}} $. Recalling the definition of $\tilde{\gamma}$ and dividing both sides by  the function $c$ (and using that $c(x)\geq c_1>0$), this leads to 
\begin{equation}\label{eq:LongEq}
	\begin{split}
		&(c(x))^{-1}\nabla \cdot x_{n+1}^{1-2s}\tilde{\gamma}(x)\nabla\tilde{ u}_N(x,x_{n+1})\\
		& =e^{iN\alpha\cdot x}\frac{z_{n+1}^{1-2s}}{N^{1-2s}}\left[N^2\left(-\sum_{j,l=1}^n\gamma_{jl}(x)\alpha_j\alpha_l+(1-2s)\frac{1}{z_{n+1}}\partial_{z_{n+1}}+\partial_{z_{n+1}}^2 \right)   \right.\\
	& \quad +iN^{3/2}\sum_{j,l = 1}^{n}2\gamma_{jl}(x)\alpha_l\partial_{z_j}+N\left(\sum_{j,l=1}^{n}\gamma_{jl}(x)\partial_{z_jz_l}^2+i\sum_{j,l=1}^n \left(\frac{\partial_{x_j}(c \gamma_{jl})(x)}{c(x)}\right)\alpha_l \right)\\
	&\quad \left.+N^{1/2}\sum_{j,l=1}^{n}\left(\frac{\partial_{x_j}(c\gamma_{jl})(x)}{c(x)} \right)\partial_{z_l} \right]\sum_{r=0}^{2k}N^{-\frac{r}{2}}v_r(x,x_{n+1}).
	\end{split}
\end{equation}
By assumption, $ \gamma,\:c\in C^{2k} $, and a componentwise Taylor expansion in  the rescaled variable $ z $  shows that as $ N\to\infty $, we have uniformly for $ z\in\{z\in\R^n: \ |z|\leq 1\} $,
\[\begin{split}
	\gamma(x)&=\sum_{|\beta|<2k}\frac{N^{-\frac{|\beta|}{2}}}{\beta!}\partial_{x^\beta}^{|\beta|}\gamma(0)z^\beta+o(N^{-k}),\\
	\left( \frac{\partial_{x_j} (c \gamma)}{c} \right)(x)&=\sum_{|\beta|\leq 2k-1}\frac{N^{-\frac{|\beta|}{2}}}{\beta!}\partial_{x^\beta}^{|\beta|} \frac{\partial_{x_j}(c\gamma)(x)}{c(x)} \Big|_{x=0} z^\beta+o(N^{-k+\frac{1}{2}}),
\end{split} \]
where $ \beta\in \N^n $ is a multi-index of length $ |\beta| $. Substituting these expansions of $ \gamma(x),\frac{\partial_{x_j} (c \gamma)}{c}(x) $ into \eqref{eq:LongEq} and grouping terms with the same order in $ N $, we obtain 
\begin{align}\label{eq:HierarchyODE}
\begin{split}
	&(c(x))^{-1}\nabla\cdot x_{n+1}^{1-2s}\tilde{\gamma}(x)\nabla\tilde{ u}_N(x,x_{n+1})\\
	& \quad =e^{iN\alpha\cdot x}\frac{z_{n+1}^{1-2s}}{N^{1-2s}}\left(L_2+L_{3/2}+L_1+\cdots+L_{2-k}+L_R \right)\sum_{r=0}^{2k}N^{-\frac{r}{2}}v_r(x,x_{n+1}),
	\end{split}
\end{align}
where the  second-order differential operators $ L_2,L_{3/2},\ldots,L_{2-k}, L_R $ are given by
\[\begin{split}
	L_2&=N^2\left(-\sum_{j,\ell=1}^{n}\gamma_{jl}(0)\alpha_j\alpha_\ell+(1-2s)\frac{1}{z_{n+1}}\partial_{z_{n+1}}+\partial_{z_{n+1}}^2 \right),\\
	L_{3/2}&=N^{3/2}\left(2i\sum_{j,\ell=1}^n\gamma_{j\ell}(0)\alpha_j\partial_{z_\ell}-\sum_{j,\ell=1}^n\sum_{h=1}^n \partial_{x_h}(\gamma_{j\ell})(0)z_h\alpha_j\alpha_\ell \right),\\
	L_m&=N^m\sum_{|\mu|\leq 2}P_{m,\mu}(z)\partial_{z^\mu}^{|\mu|}\quad\text{for } m=1,\frac{1}{2},\ldots,2-k,\\
	L_R&=\sum_{|\mu|\leq 2}o(N^{2-k})\partial_{z^\mu}^{|\mu|}, \text{ as }N\to\infty, \text{ uniformly in }\{|z|\leq 1\}.
\end{split} \]
Here the functions $ P_{m,\mu}(z) $ are polynomials in $ z $ whose coefficients are determined from $ \alpha $, $ \partial_{x^\beta}^{|\beta|}\gamma(0) $ and $ \p^{|\beta|}_{x^\beta}\left(\p_{x_j}(c\gamma)/c \right)(0) $, and we interpret $ \partial^{|(0,\ldots,0)|}_{z^{(0,\ldots,0)}}= \text{Id}$. Note that the operator $ L_2 $ is an ordinary differential operator which only differentiates in the $ z_{n+1} $ variable and $ L_{3/2}, L_1,\ldots, L_{2-k} $ only act on the $ z $ variable. Now expanding \eqref{eq:HierarchyODE} and grouping terms with the same order in $N$, we obtain
\[\begin{split}
	(c(x))^{-1}\nabla&\cdot x_{n+1}^{1-2s}\tilde{\gamma}(x)\nabla\tilde{ u}_N =e^{iN\alpha\cdot x}\frac{z_{n+1}^{1-2s}}{N^{1-2s}}\left[ L_2v_0+(L_2N^{-\frac{1}{2}}v_1+L_{3/2}v_0)\right.\\
	&+(L_2N^{-1}v_2+L_{3/2}N^{-\frac{1}{2}}v_1+L_1v_0)+\cdots +\\
	&+\left.(L_2N^{-k}v_{2k}+\cdots+L_{2-k}v_0 )  \right]  +o(N^{1-k+2s})e^{iN\alpha\cdot x}z_{n+1}^{1-2s}.
\end{split} \] 
For part (c) of the proposition to be satisfied, we solve the following ODEs (of the form $ L_2f=g $) iteratively:
\[\begin{split}
	L_2v_0&=0,\quad L_2N^{-\frac{1}{2}}v_1+L_{3/2}v_0=0,\quad L_2N^{-1}v_2+L_{3/2}N^{-\frac{1}{2}}v_1+L_1v_0=0, \\
	&\cdots,\: L_2N^{-k}v_{2k}+\cdots+L_{2-k}v_0=0.
\end{split} \]

\emph{Step 2: Iterative solution of the hierarchy of ODEs, construction of the functions $v_r$.}

We solve the ODEs from above by applying Lemma \ref{lem:normal} (a) to $ L_2v_0=0 $ and Lemma \ref{lem:normal} (b) to the inhomogeneous ODEs. We complement these with boundary conditions by imposing that the functions $v_r$, $r\geq 1$, vanish on $ \{z_{n+1}=0\} $ and that they decay exponentially at infinity. Lemma \ref{lem:normal} allows us to obtain the boundary behaviour of the solutions to these ODEs in a more precise manner.  

In light of part (b) of the proposition, we set
\[
v_0(z,z_{n+1})=\eta(z)(C_0z_{n+1})^sK_s(C_0z_{n+1})
 \]
as also stated in part (a) of the proposition.  It follows that $ v_0|_{\{x_{n+1}=0\}}(x)=\bar{c}_s\eta_N(x) $ since $ K_s(t)\sim \frac{\bar{c}_s}{t^s} $ as $ t\to 0 $. 
We now proceed by induction to obtain the functions $ v_r $. Let $ r\in\{0,1,\ldots,2k-1\} $ and suppose that we are given $ \{v_0,v_1,\ldots,v_r\} $ in the form stated in part (a) of the proposition with the claimed asymptotic behaviour for the functions $ \{w_j^{\ell,\rho}\}_{j=1,\ldots,r,\:\ell=1,\ldots,j} $ for $\rho \in \{1,\dots, \binom{r-1}{\ell-1}\}$. We will then obtain the function $ v_{r+1} $ by solving 
\[L_2N^{-\frac{r+1}{2}}v_{r+1}+L_{3/2}N^{-\frac{r}{2}}v_r+L_1N^{
\frac{1-r}{2}}v_{r-1}+\cdots+L_{\frac{3-r}{2}}v_0=0,\]
which we rewrite as

\begin{equation}\label{eq:v_r+1ODE}
\begin{split}
	\big(z_{n+1}\partial_{z_{n+1}}^2+(1-2s)\partial_{z_{n+1}}-C_0^2z_{n+1} \big)&v_{r+1}=z_{n+1}\frac{L_{3/2}}{N^{3/2}}v_r\\
	&+z_{n+1}\frac{L_1}{N}v_{r-1}+\cdots+z_{n+1}\frac{L_{\frac{3-r}{2}}}{N^{\frac{3-r}{2}}}v_0.
\end{split}	
\end{equation}
We next exploit the induction hypothesis and expand the functions $ v_0,\ldots,v_r $ in \eqref{eq:v_r+1ODE}, using the formulas for $ v_0,\ldots,v_r $ given in part (a) of the proposition and regroup the terms by combining contributions which asymptotically behave as $ O(e^{-z_{n+1}}z_{n+1}^{\frac{1}{2}+\ell-1}) $ as $ z_{n+1}\to\infty $ in the normal direction.  To this end, for each $ \ell $, we relabel the indices temporarily and write 
\[\bigcup_{i=\ell}^{r}\bigcup_{\rho=1}^{\binom{i-1}{\ell-1}}\{w_{i}^{\ell,\rho}\}=:\{w^{\ell,1},w^{\ell,2},\ldots,w^{\ell,\binom{r}{\ell}} \}, \]
where we used $ \sum_{i=\ell}^{r}\binom{i-1}{\ell-1}=\binom{r}{\ell} $. The idea here is to regroup terms with the same asymptotic behaviour on the right hand side of \eqref{eq:v_r+1ODE}. This allows us to apply  Lemma \ref{lem:normal} to solve the ODE \eqref{eq:v_r+1ODE}.
With this relabelling, we obtain 
\begin{align}
\label{eq:expansion_relabelled}
\begin{split}
\big(z_{n+1}\partial_{z_{n+1}}^2+(1-2s)\partial_{z_{n+1}}-C_0^2z_{n+1} \big) v_{r+1}&=:P_{r+1}^{1,1}(z)(C_0z_{n+1})^{s}K_s(C_0z_{n+1})\\
& \quad +\sum_{\ell=1}^{r}\sum_{\rho=1}^{\binom{r}{\ell}}P_{r+1}^{\ell+1,\rho}(z)(C_0z_{n+1})^sw^{\ell,\rho}(C_0 z_{n+1}).
\end{split} 
\end{align}
Here, we  have $ P_{r+1}^{1,1}(z):=\frac{L_{\frac{3-r}{2}}(\eta)}{N^{\frac{3-r}{2}}}(z) $ and the other smooth functions $ P_{r+1}^{\ell,\rho} $ are defined  by applying the appropriate differential operators (that is, one of the following: $ L_{3/2}/N^{3/2}$,$\ldots$, $L_{2-r/2}/N^{2-r/2} $) to the smooth functions $ P_r^{\ell,\rho}(z) $.
Again, we first use the ansatz provided in the proof of Lemma \ref{lem:normal} (a) to reduce the ODE \eqref{eq:expansion_relabelled} to an ODE of the form \eqref{eq:normal1}. Then we apply Lemma \ref{lem:normal} (b) to obtain a solution of the above ODE,
\[\begin{split}
v_{r+1}(x,x_{n+1})&=\sum\limits_{\ell=1}^{r+1}\sum_{\rho=1}^{\binom{r}{\ell-1}} P_{r+1}^{\ell,\rho}(z) (C_0 z_{n+1})^s w_{r+1}^{\ell,\rho}(C_0 z_{n+1})
\end{split} \]
where $ w_{r+1}^{\ell,\rho} $ are solutions of the following ODEs provided by Lemma  \ref{lem:normal}
\[\begin{split}
t^2(w_{r+1}^{1,1})''(t)+t(w_{r+1}^{1,1})'-(s^2+t^2)w_{r+1}^{1,1}(t)&=t^2K_s(t),\\
t^2(w_{r+1}^{\ell,\rho})''(t)+t(w_{r+1}^{\ell,\rho})'-(s^2+t^2)w_{r+1}^{\ell,\rho}(t)&=t^2w^{\ell-1,\rho}(t),\\
&\text{ for }\ell=2,\ldots,r+1,\:\rho=1,\ldots,\binom{r}{\ell-1}.
\end{split} \]
Let us comment on the roles played by the indices $ r,\ell $ and $ \rho $ in the definition of the functions $ v_r $. The index $ r $ is the induction index  for which we iterate to solve the hierarchy of ODEs; the index $ \ell $ indicates the asymptotic behavior of the functions $ w^{\ell,\rho}_r $ (as shown in the part (a) of the Proposition) and the index $ \rho $ counts how many of these functions with the same asymptotic behavior is present in the formula for $ v_r $.
By our relabelling above, for any $ \rho=1,\ldots,\binom{r}{\ell-1} $, we have $ w^{\ell-1,\rho}=O(e^{-t}t^{\frac{1}{2}+\ell-2}) $. Therefore, Lemma \ref{lem:normal} (b) shows that $ w_{r+1}^{\ell,\rho}=O(e^{-t}t^{\frac{1}{2}+\ell-1}) $. Similarly, Lemma \ref{lem:normal} also guarantees the asymptotic behaviour as claimed in part (a) of the proposition.  This completes the induction step.

\emph{Step 3: Proof of the asymptotic properties from (a)-(c).} The claims in part (a) are now consequences of Lemma \ref{lem:normal} since all the functions $ w_r^{\ell,\rho} $ are ODE solutions provided by that lemma,  and part (c) is satisfied by our construction (we recall that the the aim of achieving the bounds in (c) was the reason for which we derived the above hierarchy of ODEs in the first place). To see part (b), note that $ v_r(x,x_{n+1})\to 0 $ as $ x_{n+1}\to 0 $ uniformly in $ \{|x|\leq 1/\sqrt{N}\} $ since $ t^sw_r^{\ell,\rho}(t)=o(1) $ as $ t\to 0 $ for all $ \ell\in\{1,\ldots,r\} $, $ \rho\in\{1,\ldots,\binom{r-1}{\ell-1}\} $ and $ r\in\{1,\ldots,2k\} $. Therefore, 
\[\tilde{u}_N(x,0)=e^{iN\alpha\cdot x}v_0(x,0)= \bar{c}_se^{iN\alpha\cdot x}\eta_N(x).\]

\end{proof}
\subsection{Error bounds and proof of Theorem \ref{thm:Eucl}}
\label{sec:error_bounds_anisotropic}
Concluding the proof of Theorem \ref{thm:Eucl}, we carry out the error estimates associated with the approximate solution from Proposition \ref{prop:DtN_approx}, measuring the deviation from being an exact solution. To this end, we introduce the definition of the error of $\tilde{u}_N$ of being a solution. Indeed, we define $\tilde{w}_N$ to be the unique $\dot{H}^{1}(\R^{n+1}_+, x_{n+1}^{1-2s})$ solution of the bulk equation with the boundary data $\phi_N$ (see Proposition \ref{prop:well-posed} for well-posedness):
\begin{align*}
\nabla \cdot x_{n+1}^{1-2s} \tilde{\gamma} \nabla \tilde{w}_N &= 0 \mbox{ in } \R^{n+1}_+,\\
\tilde{w}_N & = \phi_N \mbox{ on } \R^n \times \{0\}.
\end{align*}
We define the difference 
\begin{align}
\label{eq:error}
r_N(x,x_{n+1}):= \tilde{w}_N(x, x_{n+1}) - \tilde{u}_N(x,x_{n+1}),
\end{align}
as the error quantifying how well the function $\tilde{u}_N$ approximates the true solution $\tilde{w}_N$. Observe that (in the weak sense)
\begin{equation}\label{eq:remindereq}
	\begin{split}
		\nabla\cdot x_{n+1}^{1-2s}\tilde{\gamma}(x)\nabla r_N&=-\nabla\cdot x_{n+1}^{1-2s}\tilde{\gamma}(x)\nabla \tilde{u}_N \text{ in }\R^{n+1}_+,\\
		r_N&=0 \text{ on }\R^n\times\{0\}.
	\end{split}
\end{equation}
Let us fix a cut-off function on the normal direction: let $ \zeta(x_{n+1})\in C^\infty_c([0,\infty)) $ be such that 
\[ \zeta(x_{n+1})=1\text{ for }x_{n+1}\in[0,1/2],\quad \zeta(x_{n+1})=0\text{ for }x_{n+1}\geq 1, \]
and let $ \zeta_N(x_{n+1}):=\zeta(\sqrt{N}x_{n+1}) $. Using this notation, we turn to the proof of Theorem \ref{thm:Eucl}.

\begin{proof}[Proof of Theorem \ref{thm:Eucl}]
We present the proof under the assumption  that $ \gamma,c\in C^{2k} $ with $ k\geq 1 $ so that we can keep track of the asymptotics in $ N $ more precisely. 

\emph{Step 1:} The strategy of the proof of Theorem \ref{thm:Eucl} is to understand the dominant terms in the expansion of $\inp{\Lambda_\gamma \phi_N,\phi_N}_{H^{-s}, H^s}$ as $ N\to\infty $. To this end, we use the definition of the Dirichlet-to-Neumann map and split the contribution as follows
\[\begin{split}
	\inp{\Lambda_\gamma \phi_N, \phi_N}_{H^{-s}(\R^n), H^s(\R^n)}&=\int_{\R^{n+1}_+}x_{n+1}^{1-2s}\tilde{\gamma}\nabla \tilde{w}_N\cdot \nabla (\zeta_N\ol{\tilde{u}_N})dxdx_{n+1}\\
	&=\int_{\R^{n+1}_+}x_{n+1}^{1-2s}\tilde{\gamma}\nabla \tilde{u}_N\cdot\nabla(\zeta_N\ol{\tilde{u}_N})+\int_{\R^{n+1}_+}x_{n+1}^{1-2s}\tilde{\gamma}\nabla r_N\cdot\nabla (\zeta_N\ol{\tilde{u}_N})\\
	&=:I+II.
\end{split} \]
Building on this decomposition, we will show that the leading order term in $ N $ contains information about $ c(0)^{\frac{1}{2s}}C_0 $ where $  C_0=\sqrt{\sum_{j,l=1}^n \gamma_{jl}(0)\alpha_j\alpha_l} $.

\emph{Step 2: Estimate of $ \nabla\tilde{u}_N $.} To estimate $ I $ and $ II $, we will show that 
\begin{equation}\label{eq:approxsolgradient}
	\nabla\tilde{u}_N(x,x_{n+1})=e^{iN\alpha\cdot x}\begin{pmatrix}
		\alpha iN \eta_N(x)(C_0Nx_{n+1})^sK_s(C_0Nx_{n+1})+R_N^1(x,x_{n+1}) \\
		-\eta_N(x)(C_0N)^{s+1}x_{n+1}^sK_{s-1}(C_0Nx_{n+1})+R_N^2(x,x_{n+1})
	\end{pmatrix},
\end{equation}
where 
\[\begin{split}
	R_N^1(x,x_{n+1})&=i\alpha\sum_{r=1}^{2k}N^{1-\frac{r}{2}}\sum_{\ell=1}^r\sum_{\rho=1}^{\binom{r-1}{\ell-1}}P_r^{\ell,\rho}(\sqrt{N}x)(C_0Nx_{n+1})^sw_r^{\ell,\rho}(C_0Nx_{n+1})\\
	&\quad+\sqrt{N}(\nabla_x\eta)(\sqrt{N}x)(C_0Nx_{n+1})^sK_s(C_0Nx_{n+1})\\
	&\quad+\sum_{r=1}^{2k}N^{\frac{1}{2}-\frac{r}{2}}\sum_{\ell=1}^{r}\sum_{\rho=1}^{\binom{r-1}{\ell-1}}(\nabla_xP_r^{\ell,\rho})(\sqrt{N}x)(C_0Nx_{n+1})^sw_r^{\ell,\rho}(C_0Nx_{n+1}),\\
	R_N^2(x,x_{n+1})&=\sum_{r=1}^{2k}N^{-\frac{r}{2}}\sum_{\ell=1}^{r}\sum_{\rho=1}^{\binom{r-1}{\ell-1}}P_r^{\ell,\rho}(\sqrt{N}x)\frac{d}{dx_{n+1}}\left((C_0Nx_{n+1})^sw_r^{\ell,\rho}(C_0Nx_{n+1}) \right).
\end{split}
\]
For the estimates in Step 3 below, we set $ x=z/\sqrt{N} $ and $ x_{n+1}=z_{n+1}/N $ in the above, and then apply Lemma \ref{prop:DtN_approx} (a) (using the asymptotics for $ w^{\ell,\rho}_r $ to estimate $ R_N^1 $ and the asymptotics for $\frac{d}{dt}((At)^sw_r^{\ell,\rho}(At))$ to estimate $ R_N^2 $) to obtain 
\begin{equation}\label{eq:RN1RN2}
	\begin{split}
		R_N^1\left(\frac{z}{\sqrt{N}},\frac{z_{n+1}}{N} \right)&=\sqrt{N}O(1)\text{ as }z_{n+1}\to 0, \\
		R_N^1\left(\frac{z}{\sqrt{N}},\frac{z_{n+1}}{N} \right)&=\sqrt{N}O\left(z_{n+1}^{s+2k-\frac{1}{2}}e^{-C_0z_{n+1}}\right) \text{ as }z_{n+1}\to \infty,\\
		R_N^2\left(\frac{z}{\sqrt{N}},\frac{z_{n+1}}{N} \right)&= \sqrt{N}O(z_{n+1}^{2s-1}) \text{ as }z_{n+1}\to 0,\\
		R_N^2\left(\frac{z}{\sqrt{N}},\frac{z_{n+1}}{N} \right)&= \sqrt{N}O\left(z_{n+1}^{s+2k-\frac{1}{2}}e^{-C_0z_{n+1}}\right) \text{ as }z_{n+1}\to \infty,
	\end{split}
\end{equation}
where the implicit constants in the above are independent of $ N $.

\emph{Step 3: Estimate of $ I $.} Let $ D_N:=\{|x|\leq \frac{1}{\sqrt{N}}\} $ be the region that contains the support of $ \phi_N $. As $\zeta_N(x_{n+1})\equiv 1$ for $ x_{n+1}\in[0,\frac{1}{2\sqrt{N}}] $, we write 
\[
\begin{split}
	I&=\int_{D_N}\int_{0}^{\frac{1}{2\sqrt{N}}}x_{n+1}^{1-2s}\tilde{\gamma}\nabla\tilde{u}_N\cdot\ol{\nabla\tilde{u}_N}+\int_{D_N}\int_{\frac{1}{2\sqrt{N}}}^{1}x_{n+1}^{1-2s}\tilde{\gamma}\nabla\tilde{u}_N\cdot \nabla(\zeta_N\ol{\tilde{u}_N}) \\
	&=:I_1+I_2.
\end{split} \]
For $ I_2 $, recalling Lemma \ref{lem:Bessel} (d), Proposition \ref{prop:DtN_approx} (a) and \eqref{eq:RN1RN2}, we note that both $ \tilde{u}_N $ and $ |\nabla\tilde{u}_N| $ decay exponentially as $ x_{n+1}\to\infty $. Hence, it follows that
\[|I_2|=O(e^{-C_0\frac{\sqrt{N}}{2}})\text{ as }N\to\infty. \]
For $ I_1 $, by \eqref{eq:approxsolgradient} and the definition of $ \tilde{\gamma} $, we compute
\[I_1=I_3+I_4+I_5+I_6+I_7, \]
where 
\[ \begin{split}
	I_3&=\int_{D_N}\int_{0}^{\frac{1}{2\sqrt{N}}}x_{n+1}^{1-2s}N^2\eta_N^2(x)(C_0Nx_{n+1})^{2s}K_s^2(C_0Nx_{n+1})c(x)\gamma(x)\alpha\cdot \alpha dxdx_{n+1},\\
	I_4&=\int_{D_N}\int_{0}^{\frac{1}{2\sqrt{N}}}x_{n+1}\eta_N^2(x)c(x)(C_0N)^{2s+2}K_{s-1}^2(C_0Nx_{n+1})dxdx_{n+1},\\                |I_5|&\leq2\left|\int_{D_N}\int_{0}^{\frac{1}{2\sqrt{N}}}N\eta_N(x)(C_0Nx_{n+1})^sK_s(C_0Nx_{n+1})\gamma(x)\alpha\cdot R_N^1(x,x_{n+1})x_{n+1}^{1-2s}dxdx_{n+1} \right|,\\                
	|I_6|&\leq 2\left|\int_{D_N}\int_{0}^{\frac{1}{2\sqrt{N}}}\eta_N(x)(C_0N)^{s+1}K_{s-1}(C_0Nx_{n+1}) R_N^2(x,x_{n+1}) x_{n+1}^{1-s}dxdx_{n+1}    \right|,    \\
	I_7&=\int_{D_N}\int_{0}^{\frac{1}{2\sqrt{N}}}x_{n+1}^{1-2s}\left(\gamma(x) R_N^1(x,x_{n+1})\cdot\ol{R_N^1(x,x_{n+1})}+|R_N^2(x,x_{n+1})|^2 \right)dxdx_{n+1}.
\end{split} \]
For $ I_3 $, substituting in the Taylor expansion $ c(x)\gamma(x)=c(0)\gamma(0)+\sum_{1\leq|\beta|\leq 2k}c_{\beta}N^{-\frac{|\beta|}{2}}\frac{z^\beta}{\beta!}+o(N^{-k}) $ and carrying out a change of variables in the higher order terms, we obtain 
\[\begin{split}
	I_3&=\int_{D_N}\int_{0}^{\frac{1}{2\sqrt{N}}}x_{n+1}^{1-2s}N^2\eta_N^2(x)(C_0Nx_{n+1})^{2s}K_s^2(C_0Nx_{n+1})c(0)\gamma(0)\alpha\cdot\alpha\\
	&\quad+O(N^{\frac{1}{2}+2s-\frac{n}{2}-1})\int_{|z|<1}\int_{0}^{\frac{\sqrt{N}}{2}}\eta^2(z)z_{n+1}K_s^2(C_0z_{n+1})dzdz_{n+1}.
\end{split}
\]
By the asymptotic behaviour of $ K_s $ in Lemma \ref{lem:Bessel}, we obtain 
\begin{align*}
 \int_{|z|<1}\int_{0}^{\infty}\eta^2(z)z_{n+1}K_s^2(C_0z_{n+1})dzdz_{n+1}<\infty . 
\end{align*} 
 Therefore, performing a change of variables as in \eqref{eq:scaling}, we deduce that
\[\begin{split}
	I_3&=O(N^{\frac{1}{2}+2s-\frac{n}{2}-1})+c(0)\gamma(0)\alpha\cdot\alpha C_0^{2s}N^{1+2s-\frac{n}{2}-1}\int_{|z|<1}\int_{0}^{\frac{\sqrt{N}}{2}}z_{n+1}\eta^2(z)K_s^2(C_0z_{n+1})dzdz_{n+1}\\
	&=O(N^{-\frac{1}{2}+2s-\frac{n}{2}})+ c(0)C_0^{2s+2}N^{2s-\frac{n}{2}}\frac{1}{C_0^2}\int_{0}^{C_0\frac{\sqrt{N}}{2}}z_{n+1}K_s^2(z_{n+1})dz_{n+1},
\end{split}
\]
where in the last line we have used $ \int_{|z|<1}\eta^2(z)dz=1 $ and the definition of $ C_0 $. Again by Lemma \ref{lem:Bessel} and the fact that $ 1-2s>-1 $, we find
\[c_1:=\int_{0}^{\infty}z_{n+1}K_s^2(z_{n+1})dz_{n+1}<\infty. \]
Therefore, we conclude that 
\begin{equation}\label{eq:I3asym}
	\lim_{N\to\infty}\frac{I_3}{N^{2s-\frac{n}{2}}}=\lim_{N\to\infty} \left(\frac{O(N^{-\frac{1}{2}+2s-\frac{n}{2}})}{N^{2s-\frac{n}{2}}}\right)+c(0)c_1C_0^{2s}=c_1 c(0) C_0^{2s}.
\end{equation}
Similarly, for $ I_4 $,  first substituting in $ c(x)=c(0)+O(N^{-\frac{1}{2}}) $ in $ D_N $, secondly changing variables as in \eqref{eq:scaling} and then considering the transformation $ z_{n+1}\mapsto \frac{z_{n+1}}{C_0} $, we are led to 
\[I_4=N^{2s-\frac{n}{2}}c(0)C_0^{2s}\int_{0}^{C_0\frac{\sqrt{N}}{2}}z_{n+1}K^2_{s-1}(z_{n+1})dz_{n+1} + O(N^{2s-\frac{1}{2}-\frac{n}{2}}).\]
Hence, if we define 
\[c_2:=\int_{0}^{\infty} z_{n+1}K^2_{s-1}(z_{n+1})dz_{n+1},\]
we find that
\begin{equation}\label{eq:I4asym}
	\lim_{N\to\infty}\frac{I_4}{N^{2s-\frac{n}{2}}}=c_2 c(0)C_0^{2s}.
\end{equation}
Note that $ c_2<\infty $ by Lemma \ref{lem:Bessel}. Next we show that the quantities $ I_5, I_6,I_7 $ are of order $ o(N^{2s-\frac{n}{2}}) $ as $ N\to\infty $. We deal with these terms by first changing variables to $ z,z_{n+1} $ and then using the asymptotic relations given by \eqref{eq:RN1RN2} and Lemma \ref{lem:Bessel}. 
For $ I_5 $, 
\[\begin{split}
	|I_5|&\leq C N^{-\frac{n}{2}-1}\int_{|z|<1}\int_{0}^{\frac{\sqrt{N}}{2}}N\eta(z)(C_0z_{n+1})^sK_s(C_0z_{n+1})\left|R_N^1\left(\frac{z}{\sqrt{N}}, \frac{z_{n+1}}{N} \right)\right| N^{2s-1}z_{n+1}^{1-2s}\\
	&\leq C N^{-\frac{n}{2}+2s-1}\int_{|z|<1}\int_{0}^{\infty}\eta(z)(z_{n+1})^{1-s}K_s(C_0z_{n+1})\left|R_N^1\left(\frac{z}{\sqrt{N}}, \frac{z_{n+1}}{N} \right)\right|dzdz_{n+1},
\end{split} \] 
for some $C>0$.
By the asymptotic behaviour of $ K_s $ and $ R_N^1 $ as $ z_{n+1}\to 0 $,  the integrand in the above expression is of order $ O(1) $, so it is integrable around $ z_{n+1}=0 $. As $ N\to \infty $, by \eqref{eq:RN1RN2},
\[|I_5|=O\left(N^{-\frac{1}{2}-\frac{n}{2}+2s} \right), \]
from which we deduce that 
\begin{equation}\label{eq:I5asym}
	\lim_{N\to\infty}\frac{|I_5|}{N^{2s-\frac{n}{2}}}=0.
\end{equation}
Similarly, the change of variables \eqref{eq:scaling} shows that for some $C>0$
\[\begin{split}
	|I_6|&\leq CN^{2s-\frac{n}{2}-1} \int_{|z|<1}\int_{0}^{\infty}\eta(z)K_{s-1}(C_0z_{n+1})\left|R_N^2(z/\sqrt{N},z_{n+1}/N) \right|z_{n+1}^{1-s}dzdz_{n+1},\\
	|I_7|&\leq CN^{2s-\frac{n}{2}-2}\int_{|z|<1}\int_{0}^{\infty}z_{n+1}^{1-2s}\left( \left|R_N^1\left(\frac{z}{\sqrt{N}},\frac{z_{n+1}}{N}\right) \right|^2+ \left|R_N^2\left(\frac{z}{\sqrt{N}},\frac{z_{n+1}}{N}\right) \right|^2  \right)dzdz_{n+1}.
\end{split} \]
Now we estimate $ |I_6| $ by Lemma \ref{lem:Bessel} (in particular the estimate for $ K_{s-1} $) and \eqref{eq:RN1RN2}, and estimate $ |I_7| $ by Lemma \ref{lem:Bessel} (c)(d) and \eqref{eq:RN1RN2}. This leads to
\[|I_6|=O\left(N^{-\frac{1}{2}-\frac{n}{2}+2s} \right),\quad |I_7|=O\left(N^{-1-\frac{n}{2}+2s} \right),\] 
which implies 
\begin{equation}\label{eq:I6I7}
	\lim_{N\to\infty}\frac{I_6}{N^{2s-\frac{n}{2}}}
	=\lim_{N\to\infty}\frac{I_7}{N^{2s-\frac{n}{2}}}=0.
\end{equation}
Combining the estimates for $ I_2 $, \eqref{eq:I3asym}, \eqref{eq:I4asym}, \eqref{eq:I5asym} and \eqref{eq:I6I7}, we obtain
\[\lim_{N\to\infty}\frac{I}{N^{2s-\frac{n}{2}}}= c(0)C_0^{2s}(c_1+c_2). \]
This concludes the discussion of the asymptotic behaviour of the contribution $I$.

In order to estimate the energy of the approximate solution, we note that if we replace $ I_2 $ by 
\[ \tilde{I}_2=\int_{D_N}\int_{\frac{1}{2\sqrt{N}}}^{\infty}x_{n+1}^{1-2s}|\nabla\tilde{u}_N|^2 dx dx_{n+1},  \]
then, by the exponential decay, we find for some $ p,q>0 $, as $ N\to\infty $,
\[ \tilde{I}_2=O\left( N^p\int_{\frac{\sqrt{N}}{2}}^{\infty}t^qe^{-C_0t}dt \right)=O(e^{-\frac{C_0\sqrt{N}}{4}}).\]
Therefore, together with the estimate for $ I_1 $, we further obtain an energy estimate for the approximate solution, as $ N\to\infty $,
\begin{equation}\label{eq:approxsolestimate}
	\int_{D_N}\int_{0}^{\infty}|\nabla\tilde{u}_N|^2x_{n+1}^{1-2s}dxdx_{n+1}\leq \tilde{C}N^{2s-\frac{n}{2}} 
\end{equation}
for some constant $ \tilde{C} $ only depending on  $ c_2\gamma_+ $ where $ c_2 $ is the upper bound of $ c(\cdot) $ introduced in \eqref{eq:est_c} and $ \gamma_+ $ is the supremum over $ x\in\R^n $ of the largest eigenvalue of $ \gamma(x) $.

\emph{Step 4: Estimate of $ II $.} We split the integral as we did for $ I $ to read 
\[\begin{split}
	II&=\int_{D_N}\int_{0}^{\frac{1}{2\sqrt{N}}}x_{n+1}^{1-2s}\tilde{\gamma}\nabla r_N\cdot \nabla\ol{\tilde{u}_N}+\int_{D_N}\int_{\frac{1}{2\sqrt{N}}}^{1}x_{n+1}^{1-2s}\tilde{\gamma}\nabla r_N \cdot \nabla(\zeta_N\ol{\tilde{u}_N}) \\
	&=:II_1+II_2.
\end{split}
\]
To estimate $ II_2 $, by Cauchy-Schwarz and then the energy estimates for $ r_N $ (recall that $ r_N $ is a solution of \eqref{eq:remindereq}),
\[ \begin{split}
	|II_2|&\leq C\left(\int_{D_N}\int_{\frac{1}{2\sqrt{N}}}^{1} |\nabla r_N|^2x_{n+1}^{1-2s} \right)^{1/2}\left(\int_{D_N}\int_{\frac{1}{2\sqrt{N}}}^{1}|\nabla(\zeta_N\tilde{u}_N)|^2 x_{n+1}^{1-2s}  \right)^{\frac{1}{2}}\\
	&\leq C\left(\int_{D_N}\int_{0}^{\infty}|\nabla\tilde{u}_N|^2x_{n+1}^{1-2s} \right)^{1/2}\left(\int_{D_N}\int_{\frac{1}{2\sqrt{N}}}^{1}\left(|(\nabla\zeta_N)\tilde{u}_N|^2+|(\nabla\tilde{u}_N)\zeta_N|^2\right) x_{n+1}^{1-2s}  \right)^{\frac{1}{2}}\\
	&=O\left(N^{s-\frac{n}{4}} \right)O\left(e^{-\frac{C_0N^{1/2}}{4}} \right)\text{ as }N\to\infty,
\end{split} \]
where for the last line we used \eqref{eq:approxsolestimate} and exponential decay of $ \tilde{u}_N $ and $| \nabla\tilde{u}_N| $, and $ C $ is a constant only depending on $ \gamma $. 

We next estimate $ II_1 $. Integrating by parts leads to 
\[\begin{split}
	II_1&=-\int_{D_N}\int_{0}^{\frac{1}{2\sqrt{N}}}r_N \nabla\cdot\left(\tilde{\gamma}x_{n+1}^{1-2s}\nabla\ol{\tilde{u}_N} \right)dxdx_{n+1}+\int_{\Gamma_1\cup\Gamma_2\cup\Gamma_3}r_Nx_{n+1}^{1-2s}\tilde{\gamma}\nabla\ol{\tilde{u}_N}\cdot \nu d\mathcal{H}^{n-1}\\
	&=:A+B,
\end{split}  \]
where  $\Gamma_1=\{|x|=\frac{1}{\sqrt{N}},\:0\leq x_{n+1}\leq\frac{1}{2\sqrt{N}} \}$, $\Gamma_2=\{|x|\leq \frac{1}{\sqrt{N}},\:x_{n+1}=0 \}$ and  $\Gamma_3=\{|x|\leq \frac{1}{\sqrt{N}},\:x_{n+1}=\frac{1}{2\sqrt{N}} \}$ and $ \nu $ denotes the corresponding outer unit normal vector field on $ \Gamma_1,\Gamma_2,\Gamma_3 $. We first estimate $ A $:
\begin{align}
\label{eq:error_1}
\begin{split}
	|A|&=\left| \int_{D_N}\int_{0}^{\frac{1}{2\sqrt{N}}}r_Nx_{n+1}^{-\frac{1+2s}{2}} x_{n+1}^{\frac{1+2s}{2}}\nabla\cdot(\tilde{\gamma}x_{n+1}^{1-2s}\nabla\ol{\tilde{u}_N})dxdx_{n+1}\right|\\
	&\leq \left(\int_{D_N}\int_{0}^{\infty} |r_N|^2x_{n+1}^{-1-2s}\right)^{1/2}\left(\int_{D_N}\int_{0}^{\frac{1}{2\sqrt{N}}}x_{n+1}^{1+2s}|\nabla\cdot\tilde{\gamma}x_{n+1}^{1-2s}\nabla\tilde{u}_N|^2 \right)^{1/2}\\
	&\leq C_s\left(\int_{\R^{n+1}_+}|\nabla r_N|^2x_{n+1}^{1-2s} \right)^{1/2}\left( \int_{D_N}\int_{0}^{\frac{1}{2\sqrt{N}}}x_{n+1}^{1+2s}o\left(N^{2-2k+4s} \right)(Nx_{n+1})^{1-2s} \right)^{1/2},
\end{split} 
\end{align}
where we have used that $ \tilde{u}_N $ is an approximate solution in the sense given in Proposition \ref{prop:DtN_approx} (c) and
the  Hardy's inequality in the following form: for any $\phi\in C^\infty_c(\R^{n+1}_+)  $ and $ s\in(0,1) $,
\begin{equation}\label{eq:Hardy2}
	\int_{\R^{n+1}_+}\phi^2x_{n+1}^{-1-2s}dxdx_{n+1}\leq \frac{1}{s^2}\int_{\R^{n+1}_+}(\p_{n+1}\phi)^2x_{n+1}^{1-2s}dxdx_{n+1},
\end{equation} 
see Exercise 1.2.8 in \cite{Gr08}. Moreover, we used that $r_N = 0$ on $x_{n+1}=0$ in a trace sense together with density properties.
Then, using that by energy estimates and the compact support of $\tilde{u}_N$,  the first term on the right hand side of \eqref{eq:error_1} is estimated by \eqref{eq:approxsolestimate}
\begin{align*}
\int_{\R^{n+1}_+}|\nabla r_N|^2x_{n+1}^{1-2s}
\leq C\int_{D_N}\int_{0}^{\frac{1}{2\sqrt{N}}} |\nabla \tilde{u}_N|^2x_{n+1}^{1-2s} = O(N^{s-\frac{n}{4}}).
\end{align*}
 For the second term, we change variables to $ z,z_{n+1} $, which leads to
 \begin{align*}
& \int_{D_N}\int_{0}^{\frac{1}{2\sqrt{N}}}x_{n+1}^{1+2s}o\left(N^{2-2k+4s} \right)(Nx_{n+1})^{1-2s} dx dx_{n+1}\\
& = o(N^{2-2k+4s})\left(\int_{|z|<1}\int_{0}^{\frac{\sqrt{N}}{2}}N^{-\frac{n}{2}-2s-2}z_{n+1}^2dzdz_{n+1} \right)\\
& = o(N^{-2k-2s-\frac{n}{2}})\int_{0}^{\frac{\sqrt{N}}{2}}z_{n+1}^2dzdz_{n+1}= o(N^{-2k+2s-\frac{n}{2} + \frac{3}{2}} ).
 \end{align*}
 Combining these arguments, we arrive at
\[ \begin{split}
	|A|&= O(N^{s-\frac{n}{4}})o(N^{-k+s-\frac{n}{4} + \frac{3}{4}})=o\left(N^{-\frac{n}{2}-k+2s+\frac{3}{4}} \right).
\end{split} \]
Consequently, if $ k\geq 1 $,
\[\lim_{N\to\infty}\frac{A}{N^{2s-\frac{n}{2}}}= \lim_{N\to\infty} o(N^{-k+\frac{3}{4}})=0.  \]
It remains to estimate $ B $. To this end, we first observe that $ r_N\in\dot{H}^1_0(\R^{n+1}_+,x_{n+1}^{1-2s}) $. Thus, $ r_N|_{\Gamma_2}=0 $. Next, in order to deal with the contribution on $\Gamma_3$, we establish a trace estimate for $ r_N $. Suppose $ u\in C^\infty_c(\R^{n+1}_+) $, then for any $ x\in\R^n $,
\[ \begin{split}
	\left|u\left(x,\frac{1}{2\sqrt{N}} \right)\right|&=\left|\int_{0}^{\frac{1}{2\sqrt{N}}}\p_{n+1}u (x,x_{n+1})dx_{n+1}\right|=\left|\int_{0}^{\frac{1}{2\sqrt{N}}}\p_{n+1}u (x,x_{n+1})x_{n+1}^{\frac{1-2s}{2}}x_{n+1}^{-\frac{1-2s}{2}}dx_{n+1}\right|\\
	&\leq \frac{N^{-s/2}}{2^{2s+1}s} \left( \int_{0}^{\frac{1}{2\sqrt{N}}}|\p_{n+1}u|^2x_{n+1}^{1-2s}dx_{n+1} \right)^{1/2},
\end{split} \]
and therefore, by density,
\[\begin{split}
	\int_{\R^n}\left|u\left(x,\frac{1}{2\sqrt{N}}\right)\right|^2dx&\leq\left(\frac{N^{-s/2}}{2^{2s+1}s} \right)^2 \int_{\R^n}\int_{0}^{\frac{1}{2\sqrt{N}}}|\p_{n+1}u|^2x_{n+1}^{1-2s}dxdx_{n+1}
\end{split}. \]
Hence, we conclude that for $ r_N $, it holds that (recall the definition of $\Gamma_3=\{|x|\leq \frac{1}{\sqrt{N}},\:x_{n+1}=\frac{1}{2\sqrt{N}} \}$),
\[\left(\int_{\Gamma_3} |r_N|^2 \right)^{1/2}\leq C_sN^{-s/2}\|{r_N}\|_{\dot{H}^1(\R^{n+1}_+,x_{n+1}^{1-2s})}.\]
Also, by Proposition \ref{prop:DtN_approx},  $ \tilde{u}_N $ is compactly supported in $ \{|x|\leq \frac{1}{\sqrt{N}}\} $ and hence $ \tilde{u}_N|_{\Gamma_1}=0 $. Consequently, collecting the above estimates,
\[ \begin{split}
	|B|&=\left|\int_{\Gamma_3}r_Nx_{n+1}^{1-2s}\tilde{\gamma}\nabla\ol{\tilde{u}_N}\cdot \nu \right|\leq C\left(\int_{\Gamma_3} |r_N|^2 \right)^{1/2}N^{s-\frac{1}{2}}\left(\int_{\Gamma_3}|\nabla \tilde{u}_N|^2 \right)^{1/2}\\
	&= \|{r_N}\|_{\dot{H}^1(\R^{n+1}_+,x_{n+1}^{1-2s})}N^{\frac{s-1}{2}}O\left(e^{\frac{-C_0\sqrt{N}}{2}} \right)=O\left(e^{-\frac{C_0\sqrt{N}}{4}} \right).
\end{split} \]
Combining the estimates for $ II_2, A,B $, we obtain that 
\[\lim_{N\to\infty}\frac{II}{N^{2s-\frac{n}{2}}}=0 .\]
Combining the estimates for $ I $ in Step 3 and for $ II $ in Step 4, then allows us to conclude \eqref{eq:aim_PDE} and to hence deduce the desired limiting identity by specializing to the case that $c(x) = 1$.

\emph{Step 5: Normalization.}
Last but not least, the uniform bound for the rescaled data $(N^{-s + \frac{n}{4}} \phi_N)_{N \in \N}$ follows from the fact that by trace estimates (see \eqref{eq:trace}), compactness of the supports of the functions $ \phi_N $ and the energy estimate \eqref{eq:approxsolestimate} of $\tilde{u}_N$, it holds
\begin{align*}
\|N^{-s + \frac{n}{4}} \phi_N\|_{H^s(\R^n)} \leq C N^{-s + \frac{n}{4}}\| \tilde{u}_N \|_{\dot{H}^1(\R^{n+1}_+, x_{n+1}^{1-2s})} \leq \tilde{C}< \infty.
\end{align*}
This concludes the proof.
\end{proof}

In concluding this section, we discuss the proof of Corollary \ref{cor:stab_Eucl}.

\begin{proof}[Proof of Corollary \ref{cor:stab_Eucl}]
	We fix an arbitrary point $ x_0\in\R^n $. Then Theorem \ref{thm:Eucl} implies that 
	\[\begin{split}
		\lim\limits_{N \rightarrow \infty} N^{-2s+\frac{n}{2}} \inp{(\Lambda_{\gamma_1}-\Lambda_{\gamma_2})(\phi_N),\phi_N}_{H^{-s}(\R^n), H^{s}(\R^n)} = (c_1 + c_2)((C_{\alpha}(\gamma_1)(x_0))^{2s} - (C_{\alpha}(\gamma_2)(x_0))^{2s}).
	\end{split} \]
	Taking the absolute value and considering $\alpha \in \mathbb{S}^{n-1}$, the limit $ N\to\infty $ leads to 
	\begin{align}
	\label{eq:stability_1_proof}
	\begin{split}
		&|(c_1 + c_2)((C_{\alpha}(\gamma_1)(x_0))^{2s} - (C_{\alpha}(\gamma_2)(x_0))^{2s})|\\
		& =\lim_{N\to\infty} N^{-2s + \frac{n}{2}} \left|\inp{(\Lambda_{\gamma_1}-\Lambda_{\gamma_2})(\phi_N),\phi_N}_{H^{-s}(\R^n), H^s(\R^n)} \right|\\
		&\leq  \lim_{N\to\infty}\|\Lambda_{\gamma_1}-\Lambda_{\gamma_2}\|_{\mathcal{L}(H^{s}(\R^n),H^{-s}(\R^n))} N^{-2s + \frac{n}{2}} \|\phi_N\|^2_{H^s(\R^n)}\\
		& \leq C \|\Lambda_{\gamma_1}-\Lambda_{\gamma_2}\|_{\mathcal{L}(H^{s}(\R^n),H^{-s}(\R^n))}.
	\end{split}
	\end{align}
	Here we have used that the rescaled boundary data $(N^{-s + \frac{n}{4}} \phi_N)_{N \in \N} $ given by Theorem \ref{thm:Eucl} are uniformly bounded in $ H^s(\R^n) $ with respect to $ x_0$ and $N$. It remains to argue that we can pass to the metrics instead of considering the difference of the proxies $(C_{\alpha}(\gamma_j)(x_0))^{2s} $, $j \in \{1,2\}$. To this end, we first observe that the constant $C>0$ on the right hand side of \eqref{eq:stability_1_proof}  is independent of $\alpha \in \mathbb{S}^{n-1}$. As a consequence,
\begin{align*}
	\begin{split}
		&\sup\limits_{\alpha \in \mathbb{S}^{n-1}}\left|(C_{\alpha}(\gamma_1)(x_0))^{2s} - (C_{\alpha}(\gamma_2)(x_0))^{2s}\right| 
		 \leq C \|\Lambda_{\gamma_1}-\Lambda_{\gamma_2}\|_{\mathcal{L}(H^{s}(\R^n),H^{-s}(\R^n))}.
	\end{split}
\end{align*}	
Now, using the local Lipschitz continuity of the maps $[0,\infty) \ni x \mapsto x^{\frac{1}{s}}$ for $s\in (0,1)$, by the uniform boundedness of the metrics $\gamma_1, \gamma_2$ (which in turn follows from the symmetry of $\gamma_{j}$ and the assumption that $\xi \cdot \gamma_j \xi \leq C_1 |\xi|^2$ for some $0<C_1< \infty$, and for all $\xi \in \R^n$ and $j \in \{1,2\}$), we obtain
\begin{align*}
	\begin{split}
		&\sup\limits_{\alpha \in \mathbb{S}^{n-1}}\left|C_{\alpha}(\gamma_1)(x_0) - C_{\alpha}(\gamma_2)(x_0)\right|\\
		& \leq C \sup\limits_{\alpha \in \mathbb{S}^{n-1}}\left|(C_{\alpha}(\gamma_1)(x_0))^{2s} - (C_{\alpha}(\gamma_2)(x_0))^{2s}\right|
		 \leq C \|\Lambda_{\gamma_1}-\Lambda_{\gamma_2}\|_{\mathcal{L}(H^{s}(\R^n),H^{-s}(\R^n))}.
	\end{split}
\end{align*}	
Finally, due to the equivalence of norms in finite-dimensional spaces, we have that 
\begin{align*}
|\gamma_1(x_0) - \gamma_2(x_0)|
\leq C_n \sup\limits_{\ell, j \in \{1,\dots,n\}}|\gamma_{1,j\ell}(x_0) - \gamma_{2,j\ell}(x_0)|.
\end{align*}	
Hence, it remains to bound $|\gamma_{1,j\ell}(x_0) - \gamma_{2,j\ell}(x_0)|$ for $\ell, j \in \{1,\dots,n\}$. This follows by noting that for $\alpha = e_j$, we have that $C_{\alpha}(\gamma_1)(x_0) - C_{\alpha}(\gamma_2)(x_0) = \gamma_{1,jj}(x_0)- \gamma_{2,jj}(x_0)$ and for $\alpha = \frac{1}{\sqrt{2}}( e_{\ell}+ e_{j})$ we have  $C_{\alpha}(\gamma_1)(x_0) - C_{\alpha}(\gamma_2)(x_0) = \frac{1}{2}((\gamma_{1,jj}(x_0)+ \gamma_{1,\ell \ell}(x_0)+ 2\gamma_{1,\ell j}(x_0))-  (\gamma_{2,jj}(x_0)+ \gamma_{2,\ell \ell}(x_0)+ 2\gamma_{2,\ell j}(x_0)))$. As a consequence, considering such values of $\alpha \in \mathbb{S}^{n-1}$, we obtain
\begin{align*}
|\gamma_1(x_0) - \gamma_2(x_0)|
& \leq C_{n} \sup\limits_{\ell, j}|\gamma_{1,j\ell}(x_0) - \gamma_{2,j\ell}(x_0)|\\
& \leq C_{n}\sup\limits_{\alpha \in \mathbb{S}^{n-1}}\left|C_{\alpha}(\gamma_1)(x_0) - C_{\alpha}(\gamma_2)(x_0)\right|\\
& \leq C_{n,s} \|\Lambda_{\gamma_1}-\Lambda_{\gamma_2}\|_{\mathcal{L}(H^{s}(\R^n),H^{-s}(\R^n))}.
\end{align*}	
As all constants are uniform in $x_0 \in \R^n$, this concludes the proof.
\end{proof}

\subsection{Proof of the geometric versions in Theorems \ref{thm:bdry_reconstr_ext_nonloc_sol} and \ref{thm:bdry_reconstr_ext_NtD}}
\label{sec:mfd1}

In this section, we briefly indicate how Theorems \ref{thm:bdry_reconstr_ext_nonloc_sol} and \ref{thm:bdry_reconstr_ext_NtD} follow along the same lines as the result of Theorem \ref{thm:Eucl}. Indeed, both results are consequences of the extension perspective which had been introduced in \cite{CS07} for the constant and in \cite{ST10} for the variable coefficient setting. In fact, as outlined in Section \ref{sec:mfd}, it holds that $\tilde{L}_{s,O}(\phi) = c_s \sqrt{\det(g)} \lim\limits_{x_{n+1} \rightarrow 0} x_{n+1}^{1-2s} \p_{n+1} \tilde{u}^{\phi}|_{O}$ with $\tilde{u}^{\phi}$ a solution of \eqref{eq:Dirichlet_Neumann_main} and $c_s \neq 0$ the constant from the Caffarelli-Silvestre extension. We note that, when localized to a tangential coordinate patch,  the equation \eqref{eq:Dirichlet_Neumann_main} is exactly of the form as in \eqref{eq:main_one1D} with the choice $c(x)= \sqrt{\det(g)(x)}$.

Hence, after localizing in tangential directions, with an analogous argument as in the previous section, we can then construct an approximate solution $\tilde{u}_{N}$ with localized and oscillating data $\phi_N$ as above. Without loss of generality, we choose coordinates such that $x_0 =0$. As the functions $ \phi_N $ are localized, in this coordinate patch, we obtain the approximate solutions with the same properties as in Proposition \ref{prop:DtN_approx}. Building on this, with the argument from above, we turn to the proof of Theorems  \ref{thm:bdry_reconstr_ext_nonloc_sol} and \ref{thm:bdry_reconstr_ext_NtD}.

\begin{proof}[Proof of  Theorems \ref{thm:bdry_reconstr_ext_nonloc_sol} and \ref{thm:bdry_reconstr_ext_NtD}]
As discussed above, we assume that $x_0 =0$ and begin by observing that for $\phi_N$ as in the previous section,
\begin{align}
\label{eq:int_by_parts}
\begin{split}
\langle \phi_N, \frac{1}{\sqrt{\det(g)}} \tilde{L}_{s,O}(\phi_N) \rangle_{H^{s}(M,dV_g), H^{-s}(M, dV_g)} =c_s \int\limits_{M \times  \R_+}x_{n+1}^{1-2s} \nabla_{\tilde{g}} \tilde{w}_N \cdot \tilde{g}   \nabla_{\tilde{g}}(\tilde{u}_N \zeta_N) dV_g dx_ {n+1} .
\end{split}
\end{align}
Here $\nabla_{\tilde{g}}$ denotes the gradient with respect to the product metric $\tilde{g}:= g + dx_{n+1}\otimes d x_{n+1}$, $\tilde{u}_N$ is the approximate solution constructed above and $ \tilde{w}_N$ is a solution of 
\[\begin{split}
(\p_{n+1} x_{n+1}^{1-2s} \p_{n+1} + x_{n+1}^{1-2s}\Delta_g)  \tilde{w}_N&=0\text{ in }M\times\R^+\\
\tilde{w}_N&=\phi_N\text{ on  }M.
\end{split} \]
The identity \eqref{eq:int_by_parts} follows from the definition of the solution-to-source data  and its realization through a variable coefficient Caffarelli-Silvestre extension from \cite{ST10}. 
In local coordinates (and by the localization of the approximate solution $\tilde{u}_N$ this is possible), we can write
\begin{align*}
 \int\limits_{M \times  \R_+}x_{n+1}^{1-2s} \nabla_{\tilde{g}} \tilde{w}_N \cdot \tilde{g}   \nabla_{\tilde{g}}(\tilde{u}_N \zeta_N) dV_g dx_ {n+1} 
=  \int\limits_{B_r(0) \times  \R_+} x_{n+1}^{1-2s}\nabla \tilde{w}_N \cdot \tilde{\gamma}   \nabla (\tilde{u}_N \zeta_N) dx dx_ {n+1} .
\end{align*}
where $ \nabla $ on the right hand side above denotes the Euclidean gradient, and in local coordinates $ \tilde{\gamma} $ is defined to be
\[\tilde{\gamma}(x,x_{n+1})=\sqrt{\det g(x)}\begin{pmatrix}
	g^{-1}(x) & 0\\
	0 & 1
\end{pmatrix}. \]

Without loss of generality, we may assume that $r=1$.
 Also we define the error term $ r_N(x,x_{n+1}):=\tilde{w}_N(x,x_{n+1})-\tilde{u}_N(x,x_{n+1}) $, which satisfies $ \nabla_{\tilde{g}}\cdot x_{n+1}^{1-2s}\tilde{g}\nabla_{\tilde{g}} r_N=-\nabla_{\tilde{g}}\cdot x_{n+1}^{1-2s}\tilde{g}\nabla_{\tilde{g}} \tilde{u}_N $ in $ M\times\R^+ $ with $ r_N=0 $ on $ M\times\{0\} $. In particular, we note that  $ \supp\nabla (r_N) \subset B_1(0)\times  \R_+$.
Therefore, passing to local coordinate as above, 
\[\begin{split}
\langle \phi_N, \frac{1}{\sqrt{\det(g)}} \tilde{L}_{s,O}(\phi_N) \rangle_{H^{s}(M,dV_g), H^{-s}(M, dV_g)} =c_s\int_{B_1(0)\times \R_+}x_{n+1}^{1-2s}&\nabla \tilde{u}_N\cdot\tilde{\gamma}\nabla(\zeta_N\ol{\tilde{u}_N})dxdx_{n+1}\\
+c_s\int_{B_1(0)\times \R_+}&x_{n+1}^{1-2s}\nabla r_N\cdot\tilde{\gamma}\nabla(\zeta_N\ol{\tilde{u}_N})dxdx_{n+1}
\end{split}
 \]
Now in this local coordinate patch where $ x_0=0 $, with the choice of $ \tilde{\gamma} $ as above
we have reduced the problem to the Euclidean case that we considered in Sections \ref{sec:DtN_approx_prelim}-\ref{sec:error_bounds_anisotropic}. By setting $u_N:= \phi_N$ these considerations combined with Theorem \ref{thm:Eucl} then conclude the proof. 
Finally, the normalization of the rescaled boundary data again follows by the trace estimates.
\end{proof}

As a final result on the solution-to-source setting, we provide the proof of Corollary \ref{cor:metric2}.

\begin{proof}[Proof of Corollary \ref{cor:metric2}]
Let $\phi_N$ be as above.
Fix any $ x_0\in O $, Theorem \ref{thm:bdry_reconstr_ext_nonloc_sol} yields that the knowledge of 
\begin{align*}
\langle \phi_N, \frac{1}{\sqrt{\det(g)}} \tilde{L}_{s,O}(\phi_N) \rangle_{H^{s}(M,dV_g), H^{-s}(M, dV_g)}
\end{align*}
which allows us to reconstruct, in the limit $N \rightarrow \infty$ the quantity
\begin{align}
\label{eq:bilin_form}
\sum\limits_{j, \ell=1}^{n} \det(g(x_0))^{\frac{1}{2s}} g_{j \ell}^{-1}(x_0) \alpha_j \alpha_{\ell}
\end{align}
for all $\alpha \in \mathbb{S}^{n-1}$. 
In particular, since $(\phi_N)_{N \in \N}$ is localized in $O$, the quantity 
\begin{align*}
\langle \phi_N, \frac{1}{\sqrt{\det(g)}} \tilde{L}_{s,O}(\phi_N)  \rangle_{H^{s}(M,dV_g), H^{-s}(M, dV_g)}
\end{align*}
 consists only of known measurement information.

Now, since $g^{-1}$ is symmetric, the recovery of the bilinear form \eqref{eq:bilin_form} for all $\alpha \in \mathbb{S}^{n-1}$, in turn, implies the recovery of $\det(g(x_0))^{\frac{1}{2s}} g^{-1}(x_0)$. Taking the determinant of this expression, we obtain the recovery of $\det(g(x_0))^{\frac{n}{2s}-1}$. Since $s\in (0,1)$, we have that $\frac{n}{2s}-1 \neq 0$, which implies that we infer the recovery of $\det(g(x_0))$ and, by the recovery of $\det(g(x_0))^{\frac{1}{2s}} g^{-1}(x_0)$, also that of $g^{-1}(x_0)$ and $g(x_0)$. This concludes the proof.
\end{proof}

\section{The Source-to-Solution Problem  -- Boundary Reconstruction for the Anisotropic Extension Problem with Generalized Neumann-to-Dirichlet Data}
\label{sec:NtD}

The setting with Neumann-to-Dirichlet data can be treated similarly as the setting with Dirichlet-to-Neumann data by also adopting an extension perspective. Also in this context, the method from \cite{B01} and \cite{NT01} yields the symbol expansion of the operator on the boundary up to an arbitrary order. As this corresponds to the measurement setting from \cite{FGKU21} we outline the main steps and point out the main differences compared the previous section. As our first main theorem in this section, we provide the analogue of Theorem \ref{thm:Eucl} for Neumann-to-Dirichlet measurements.

As in the previous section, we first consider a closely related problem on $\R^{n+1}_+$. Moreover, similarly as in the section on the Dirichlet-to-Neumann measurements, we consider a slightly more general setting than required for the Euclidean problem in order to simultaneously also treat the geometric variant. The result on the manifold will follow similarly as in the previous section. Consider
\begin{align}
\label{eq:main_one1}
\begin{split}
\nabla \cdot x_{n+1}^{1-2s} \tilde{\gamma}(x) \nabla \tilde{u} & = 0 \mbox{ in } \R^{n+1}_+,\\
 c(x)\lim\limits_{x_{n+1} \rightarrow 0} x_{n+1}^{1-2s} \p_{n+1} \tilde{u} & = f \mbox{ on } \R^n \times \{0\},
\end{split}
\end{align}
where
\begin{align*}
\tilde{\gamma}(x) = c(x)\begin{pmatrix} \gamma(x) & 0 \\ 0& 1 \end{pmatrix}
\end{align*}
is a smooth, bounded, matrix-valued function with $0<c_1 \leq c(x)\leq c_2 < \infty$ a smooth function which only depends on the tangential variables.
We assume that ${\gamma}$ is uniformly elliptic, symmetric and that $\gamma$ only depends on the tangential variables and is independent of the normal one.
We observe that for $f\in C_c^{\infty}(B_1(x_0)) $ with $\int\limits_{\R^n} f(x )dx = 0$, we have existence and uniqueness of solutions of \eqref{eq:main_one1}. 
 Indeed, without loss of generality, we assume that $x_0 =0$ and observe that, since $\hat{f}(0)=0$, a first order expansion yields
\begin{align*}
\int\limits_{B_1(0)} |\xi|^{-2s} |\hat{f}(\xi)|^2 d\xi \leq \int\limits_{B_1(0)} |\xi|^{2-2s} |\nabla_{\xi}\hat{f}(\xi)|^2 d\xi
\leq C \|x f\|_{H^{1}(\R^n)}^2 \leq C< \infty,
\end{align*}
where we used that $2-2s \in (0,2)$.
As a consequence,
\begin{align*}
\int\limits_{\R^n} |\xi|^{-2s} |\hat{f}(\xi)|^2 d\xi
&\leq \int\limits_{B_1(0)} |\xi|^{-2s} |\hat{f}(\xi)|^2 d\xi + \int\limits_{\R^n\setminus B_1(0)} |\xi|^{-2s} |\hat{f}(\xi)|^2 d\xi\\
&\leq C \|x f\|_{H^{1}(\R^n)}^2  + \|f\|_{L^2(\R^n)}^2 \leq C \|x f\|_{H^{1}(\R^n)}^2  + \|f\|_{L^2(B_1(0))}^2 < \infty.
\end{align*}
Therefore, we infer that $f\in \dot{H}^{-s}(\R^n)$. Now, for data $f \in \dot{H}^{-s}(\R^n)$ well-posedness for \eqref{eq:main_one1} then follows from energy arguments.

\begin{lem}
\label{lem:well-posed_NtD}
Let $s\in (0,1)$ and let $f \in \dot{H}^{-s}(\R^n)$. Then there exists a unique solution $\tilde{u} \in \dot{H}^{1}(\R^{n+1}_+, x_{n+1}^{1-2s})$ of \eqref{eq:main_one1} with $\tilde{u}(x,0) \in L^{\frac{2n}{n-2s}}(\R^n)$. Moreover, it satisfies the energy estimates
\begin{align*}
\|\tilde{u}(\cdot,0)\|_{L^{\frac{2n}{n-2s}}(\R^n)} + \|\tilde{u}(\cdot,0)\|_{\dot{H}^s(\R^n)}
+ \|x_{n+1}^{\frac{1-2s}{2}} \nabla \tilde{u}\|_{L^2(\R^{n+1}_+)}
\leq C \|f\|_{\dot{H}^{-s}(\R^n)}.
\end{align*}
\end{lem}  

\begin{proof}
This follows from a combination of (Sobolev) trace embeddings and the lemma of Lax-Milgram. We refer to, for instance, \cite{FGKRSU25} for further details.
\end{proof}

We consider the following measurements encoded in the (generalized) Neumann-to-Dirichlet map
\begin{align}
\label{eq:DtN2}
\begin{split}
\tilde{\Lambda}_{\gamma}:  C^{\infty}_{x_0,\diamond}(\R^n ) \rightarrow H^{s}(\R^n),\\
f \mapsto \tilde{\Lambda}_{\gamma}(f):= \tilde{u}(x,0),
\end{split}
\end{align}
where $\tilde{u}$ is a solution of \eqref{eq:main_one1}.
With slight abuse of notation, we have set $C^{\infty}_{x_0,\diamond}(\R^n ):=\{f\in C_c^{\infty}(B_1(x_0)): \ \int\limits_{\R^n} f dx = 0\}$. In what follows, without loss of generality, we will assume that $x_0 =0$.
It is our aim to reconstruct $c^{\frac{1}{s}} \gamma: \R^n \rightarrow \R^{n\times n}$ from the information encoded in the map $\tilde{\Lambda}_{\gamma}$. As above, we formulate this as a PDE version of Theorem \ref{thm:bdry_reconstr_ext_nonloc_source}:

\thmstarstarnum{\ref{thm:bdry_reconstr_ext_nonloc_source}}
\begin{thm**}[Whole space PDE version of Theorem \ref{thm:bdry_reconstr_ext_nonloc_source}]
\label{prop:NtD}
Let $s\in (0,1)$ and let $\gamma \in C^{4}(\R^n, \R^{n\times n}_+)$ and $c\in C^{4}(\R^n)$  be as above. For each point $x_0 \in \R^n $ and $\alpha \in \mathbb{S}^{n-1}$ there exist
\begin{itemize}
\item a sequence $(N_{k,\alpha})_{k \in \N} \subset \R$ with $N_{k,\alpha} \rightarrow \infty$ as $k \rightarrow \infty$,
\item a sequence of solutions $(\tilde{u}_{N_{k,\alpha}})_{k \in \N} \subset \dot{H}^{1}(\R^{n+1}_+, x_{n+1}^{1-2s})$ of the bulk equation in \eqref{eq:main_one1} with an associated sequence of boundary data $(f_{N_{k,\alpha}})_{k \in \N}$ satisfying $\int\limits_{\R^n}f_{N_{k,\alpha}} dx = 0$ 
\end{itemize}
such that
\begin{align*}
\lim\limits_{k \rightarrow \infty} N_{k,\alpha}^{2s + \frac{n}{2}} \langle f_{N_{k,\alpha}}, \tilde{\Lambda}_{\gamma}(f_{N_{k,\alpha}}) \rangle_{H^{-s}(\R^n), H^{s}(\R^n)}   = \hat{c}_s^{-2} (c_1+c_2) (\bar{C}_{\alpha}(\gamma)(x_0))^{-2s},
\end{align*}
where $\bar{C}_{\alpha}(\gamma)(x_0)$ is the constant from \eqref{eq:constant_weighted}, $c_1 = \int\limits_{0}^{\infty} t K_s^2(t) dt$, $c_{2}= \int\limits_{0}^{\infty} t K_{1-s}^2 (t) dt$ and $ \hat{c}_s= 2^{-s}\Gamma(1-s)\neq 0 $.
The functions $f_{N_{k,\alpha}}$ are $C^1$ regular with support in $B_{1/N_{k,\alpha}}(x_0)$ and they are independent of $\gamma$. 
\end{thm**}

\subsection{Proof of Theorem \ref{prop:NtD}}

The proof of Theorem \ref{prop:NtD} is largely parallel to the proof of Theorem \ref{thm:Eucl}. As a consequence, we only comment on the main changes.

\subsubsection{Construction of the approximate solution}

As above, we assume that $ \gamma, c \in C^{2k} $ for some $ k\geq 2 $, and for the approximate solution we make the same ansatz as above and seek to construct a solution of the form
\begin{align*}
\tilde{\bar{u}}_N(x, x_{n+1}) = e^{i N \alpha \cdot x} \sum\limits_{j=0}^{2k} N^{-\frac{j}{2}} \tilde{v}_j(x, x_{n+1}),
\end{align*}
where the functions $\tilde{v}_j$ are again compactly supported and now satisfy suitable weighted Neumann data.
In order to achieve this, we again carry out the formal expansion of the bulk equation and the changes of variables outlined above. This leads to the same bulk recursion formula as in \eqref{eq:HierarchyODE}. The bulk equations are now, however, complemented with appropriate Neumann instead of Dirichlet data. This leads to a slightly modified version of the expansion, in which the approximate solution $ \tilde{u}_N $ with the desired Neumann boundary data will be obtained by multiplying with a suitable normalization factor (see part (b) of the following proposition). As above, we shall fix $ x_0=0 $ and for abbreviation, in the following proposition we use the notation 
\begin{align*}
C_0:= \sqrt{\sum\limits_{j,k=1}^n \gamma_{jk}(0)\alpha_k \alpha_j}.
\end{align*}

\begin{prop}
\label{prop:expansion_NtD}
Let $\alpha\in \mathbb{S}^n$, let $N\in \N$ and let $\gamma \in C^{2k}(\R^n, \R^{n\times n})$, $c \in C^{2k}(\R^n)$ for some $k\geq 1$. Let $\eta$ be as in \eqref{eq:eta}.
Then there exists a function of the form
\begin{align}
\label{eq:approx_NtD}
\tilde{\bar{u}}_N(x,x_{n+1})
= e^{i N \alpha \cdot x} \sum\limits_{r=0}^{2k} N^{-\frac{r}{2}} \tilde{v}_r(x,x_{n+1}) 
\end{align}
with the following properties:
\begin{itemize}
\item[(a)] The functions $\tilde{v}_r$ have the following form
\begin{align*}
\tilde{v}_0(x,x_{n+1}) &= \eta(z) (C_0 z_{n+1})^s K_s(C_0 z_{n+1}),\\
\tilde{v}_r(x,x_{n+1}) &= \sum\limits_{\ell=1}^{r}\sum_{\rho=1}^{\binom{r-1}{\ell-1}} P_r^{\ell,\rho}(z) (C_0 z_{n+1})^s w_r^{\ell,\rho}(C_0 z_{n+1})\\
&\qquad- \text{Err}_r(z) \hat{c}_s^{-1} C_0^{-2s}(C_0 z_{n+1})^s K_s(C_0 z_{n+1})\text{ for } r\in \{1,\dots,2k\},
\end{align*}
where $\mbox{Err}_r(z), P_{r}^{\ell,\rho}$ are smooth functions with compact support in $\{z\in \R^n: |z|\leq 1\}$ and for $r\in\{1,\dots,2k\}$ and $ \rho\in\{1,\ldots,\binom{r-1}{\ell-1} \} $ it holds that
\begin{align*}
	w_r^{\ell,\rho}(t) &= \frac{o(1)}{t^{s}} \mbox{ as } t \rightarrow 0,\\
	w_r^{\ell,\rho}(t) &= O(e^{-t} t^{\frac{1}{2}+(\ell-1)}) \mbox{ as } t \rightarrow \infty,\\
	\lim\limits_{t \rightarrow 0} t^{1-2s} \frac{d}{dt}\left( (At)^s w_r^{\ell,\rho}(At) \right) &= A^{2s} O(1) ,\\
	\lim_{t\to\infty}e^{At}t^{-s-\frac{1}{2}-(\ell-1)}A^{-s-\frac{1}{2}-\ell}\frac{d}{dt}\left((At)^sw_r^{\ell,\rho}(At) \right)&=O(1),
\end{align*} 
for any constant $A>0$. In the last two limits, the implicit constant is independent of $A$.
\item[(b)] It holds that
\begin{align*}
c(x)\lim\limits_{x_{n+1} \rightarrow 0} x_{n+1}^{1-2s} \p_{n+1} \tilde{\bar{u}}_N(x)
=  \hat{c}_s N^{2s} C_0^{2s}  e^{i N \alpha \cdot x} \eta_N(x) =: \hat{c}_s N^{2s} C_0^{2s}f_N(x),
\end{align*}
where $\hat{c}_s= 2^{-s}\Gamma(1-s)\neq 0$.
\item[(c)] The function $\tilde{\bar{u}}_N$ is an approximate solution in the sense that
\begin{align*}
	\nabla \cdot x_{n+1}^{1-2s} \tilde{\gamma} \nabla \tilde{\bar{u}}_N =  o(N^{1-k+2s}) e^{i N \alpha \cdot x} z_{n+1}^{1-2s} \mbox{ as } N \rightarrow \infty.
\end{align*}

\end{itemize}
\end{prop}

\begin{proof}
We argue similarly as in the Dirichlet-to-Neumann setting, adopting the proof from Proposition \ref{prop:DtN_approx} suitably. 

Indeed, by Proposition \ref{prop:DtN_approx}, we have that the function $v_0$ from Proposition \ref{prop:DtN_approx} is a solution to the equation $L_2 v=0$ from \eqref{eq:HierarchyODE}. Moreover, defining $\tilde{v}_0(x,x_{n+1}):= \frac{1}{c(x)} v_0(x,x_{n+1})$, by the derivative relations and asymptotic behaviour of the modified Bessel functions (see Lemma \ref{lem:Bessel}(b), (c)) it holds that
\begin{align*}
c(x)\lim\limits_{x_{n+1} \rightarrow 0} x_{n+1}^{1-2s} \p_{x_{n+1}} \tilde{v}_0(x,x_{n+1})
&= N^{2s}\lim\limits_{z_{n+1} \rightarrow 0} z_{n+1}^{1-2s} \p_{z_{n+1}} v_0(z,z_{n+1})\\
&= \hat{c}_{s} C_0^{2s} N^{2s} \eta_N(x) ,
\end{align*} 
as claimed above.

In order to construct the functions $\tilde{v}_{r}$ with $r \in \{1,\dots, 2k\}$, we also follow the proof from above, but ``correct'' the boundary conditions such that the functions $\tilde{v}_{r}$ satisfy
\begin{align*}
c(x)\lim\limits_{x_{n+1} \rightarrow 0} x_{n+1}^{1-2s} \p_{x_{n+1}} \tilde{v}_{r}(x,x_{n+1}) = 0
\mbox{ for all } r \in \{1,\dots, 2k\}.
\end{align*} 
As above, we argue inductively. In the first step, the function $\tilde{v}_1$ is constructed by first considering the auxiliary function
\begin{align*}
	\bar{v}_1(z, z_{n+1}) = \frac{L_{3/2}(\eta(z))}{N^{3/2}C_0^2} (C_0 z_{n+1})^{s} w_1^{1,1}(C_0 z_{n+1})
	=: \bar{P}_{1}^{1,1}(z) (C_0 z_{n+1})^{s} w_1^{1,1}(C_0 z_{n+1}),
\end{align*}
 where $w_1^{1,1}$ is a solution of
\begin{align*}
t^2 w''(t) + t w'(t) - (s^2+t^2)w(t) = t^2 K_s(t).
\end{align*}
We recall that, by construction, $\bar{P}^{1,1}_1$ is independent of $N$.
Now as discussed above (see Step 2 in the proof of Proposition \ref{prop:DtN_approx}), the function $\bar{v}_1$ is a solution of the inhomogeneous ODE $ L_2N^{-1/2}v=-L_{3/2}\tilde{v}_0 $ satisfying vanishing Dirichlet, but not vanishing Neumann conditions. Indeed, we have that (by Proposition \ref{prop:DtN_approx}(a))
\begin{align*}
	c(x)\lim\limits_{z_{n+1} \rightarrow 0} z_{n+1}^{1-2s} \p_{n+1} \bar{v}_1(z, z_{n+1})
	= c(x)c_{s}C_0^{2s} \bar{P}_{1}^{1,1}(z)=:\mbox{Err}_1(z) 
\end{align*}
for some constant $ c_{s}>0 $ given by the asymptotic behaviour in Lemma \ref{lem:normal}(b).  We, hence, ``correct'' the function $\bar{v}_1$, by subtracting a multiple of the homogeneous solution which corrects the Neumann data. More precisely, we set
\begin{align*}
	\tilde{v}_1(z, z_{n+1}) &= \bar{v}_1(z,z_{n+1}) -  \mbox{Err}_1(z) \hat{c}_s^{-1} C_0^{-2s}(C_0 z_{n+1})^s K_s(C_0 z_{n+1}),
\end{align*}
where the constant $\hat{c}_s \neq 0$ is given in part (b); it is obtained as an asymptotic constant for Bessel functions, see Lemma \ref{lem:Bessel} (b).
We see that $ \tilde{v}_1 $ now satisfies vanishing Neumann condition since by Lemma \ref{lem:normal}(b), 
\[z_{n+1}^{1-2s}\p_{z_{n+1}}\left((C_0z_{n+1})^sK_s(C_0z_{n+1}) \right)\sim \hat{c}_sC_0^{2s}\text{ as }z_{n+1}\to 0. \]
Next, to solve for $\tilde{v}_2 $ (and, analogously, also in the higher order iteration below), the key observation is that in comparison to the approximate solutions obtained in Proposition \ref{prop:DtN_approx},  the new ``correction term" only changes the definition of the function $ \bar{P}_2^{1,1}(z) $ but not the profiles of the functions in the $z_{n+1}$ variable in the definition of $ \bar{v}_2 $. This follows as it is a solution of the homogeneous ODE. More precisely, $ \bar{v}_2 $ is a solution of $ L_2N^{-1}v+L_{3/2}N^{-\frac{1}{2}}\tilde{v}_1+L_1\tilde{v}_0=0 $. Hence, by Lemma \ref{lem:normal}, $ \bar{v}_2 $ is given by 
\[\begin{split}
\bar{v}_2(z,z_{n+1})&=\left(\frac{L_1\eta(z)}{NC_0^2}-\frac{L_{3/2}\hat{c}_s^{-1}C_0^{-2s}\mbox{Err}_1}{N^{3/2}}\right)(C_0z_{n+1})^sw_2^{1,1}(C_0z_{n+1})\\
&\quad\qquad+\frac{L_{3/2}P_1^1}{N^{3/2}}(C_0z_{n+1})^sw_{2}^{2,1}(C_0z_{n+1})\\
&=:\bar{P}_2^{1,1}(z) (C_0z_{n+1})^sw_2^{1,1}(C_0z_{n+1}) + \bar{P}_{2}^{2,1}(C_0z_{n+1})^sw_2^{2,1}(C_0z_{n+1}),
\end{split} \] 
where 
\[\begin{split}
	t^2(w_2^{1,1})''(t)+t(w_2^{1,1})'-(s^2+t^2)w_2^{1,1}&=t^2K_s(t),\\
	t^2(w_2^{2,1})''(t)+t(w_2^{2,1})'-(s^2+t^2)w_2^{2,1}&=t^2w_1^{1,1}(t). 
\end{split} \]
As above, by construction, the functions $\bar{P}_2^{1,1}(z)$, $\bar{P}_{2}^{2,1}(z)$ are independent of $N$.
The new ``error term'' is given by  
\[
c(x)\lim_{z_{n+1}\to 0}z_{n+1}^{1-2s}\p_{n+1}\bar{v}_2(z,z_{n+1})=c(x)\bar{P}_2^{1,1}(z)c_2^{1,1}C_0^{2s}+c(x)\bar{P}_2^{2,1}(z)c_2^{2,1} C_0^{2s}=:\mbox{Err}_2(z), \]
where $ c_2^{1,1}, c_2^{2,1} $ are asymptotic constants given by Lemma \ref{lem:normal}. Then, similarly as for $ v_1 $, setting 
\[\tilde{v}_2(z,z_{n+1}):=\bar{v}_2(z,z_{n+1}) -  \mbox{Err}_2(z) \hat{c}_s^{-1} C_0^{-2s}(C_0 z_{n+1})^s K_s(C_0 z_{n+1})\]
suffices for $ \tilde{v}_2 $ to ensure the vanishing Neumann boundary condition.   

We proceed with the general induction step. Assuming now that we obtained the functions $ \{\tilde{v}_0,\ldots,\tilde{v}_r\} $ for some $ r\in\{1,\ldots,2k-1\} $, the same proof as in the induction step in the proof of Proposition \ref{prop:DtN_approx} gives 
\[\bar{v}_{r+1}(z,z_{n+1})=\sum\limits_{\ell=1}^{r+1}\sum_{\rho=1}^{\binom{r}{\ell-1}} \bar{P}_{r+1}^{\ell,\rho}(z) (C_0 z_{n+1})^s w_{r+1}^{\ell,\rho}(C_0 z_{n+1}), \] 
for suitable functions $\bar{P}_{r+1}^{\ell,\rho}(z) $, which are independent of $N$.
As above, it remains to ``correct'' the Neumann boundary conditions. To this end, we set 
\[\mbox{Err}_{r+1}(z):= c(x)\lim_{z_{n+1}\to 0}z_{n+1}^{1-2s}\p_{n+1}\bar{v}_{r+1}(z,z_{n+1}), \]
and define 
\[\tilde{v}_{r+1}(z,z_{n+1}):= \bar{v}_{r+1}(z,z_{n+1}) -  \mbox{Err}_{r+1}(z) \hat{c}_s^{-1} C_0^{-2s}(C_0 z_{n+1})^s K_s(C_0 z_{n+1}) \]
In particular, the functions $ w_r^{\ell,\rho}(t) $, which are obtained from Lemma \ref{lem:normal}, have exactly the same asymptotic behaviour at the boundary (that is as $ z_{n+1}\to 0 $ or as $ z_{n+1}\to\infty $) as  those in Proposition \ref{prop:DtN_approx}. This allows us to iterate Lemma \ref{lem:normal} to obtain the functions $ \tilde{v}_r $ with vanishing Neumann conditions with the same asymptotic behaviour as in Proposition \ref{prop:DtN_approx}. 
\end{proof}

While having constructed approximate solutions in parallel to the Dirichlet data setting, in order to obtain an exact solution close to this, we still need to adapt the boundary condition suitably in order to ensure the vanishing mean value of the boundary data. As above, we consider $\alpha \in \mathbb{S}^{n-1}$ arbitrary but fixed, and seek to consider Neumann boundary conditions of the form
\begin{align*}
f_N(x):= \eta(\sqrt{N}x) e^{i N \alpha \cdot x }.
\end{align*}
In contrast to before, we now, however, have to ensure a vanishing mean value of these data. To this end, we will specify the cut-off function $\eta: \R^n \rightarrow \R$ and will consider a sequence $(N_{k,\alpha})_{k \in \N} \subset \R$ such that $N_{k,\alpha} \rightarrow \infty$ as $k \rightarrow \infty$.

\begin{lem}
\label{lem:cut-off}
There exists a cut-off function $\eta \in C^{\infty}(\R^n)$ such that for all $\alpha \in \mathbb{S}^{n-1}$ the following property holds: There is a sequence $(N_{k,\alpha})_{k \in \N} \subset \R$ such that $N_{k,\alpha} \rightarrow \infty$ as $k \rightarrow \infty$ and such that
\begin{align*}
\int\limits_{\R^n} f_{N_{k,\alpha}}(x) dx:= \int\limits_{\R^n} \eta(\sqrt{N_{k,\alpha}}x) e^{i N_{k,\alpha} \alpha \cdot x } dx = 0.
\end{align*}
\end{lem}

\begin{proof}
We consider $\tilde{\eta} = \chi_{[-1/2,1/2]^n}$ the characteristic function of the $n$-dimensional cube of side length one centred at zero. For $\rho$ a standard mollifier we then define
$\eta := \rho \ast \tilde{\eta}$ and note that 
\begin{align}
\label{eq:FT}
\hat{\eta}(\alpha) = \hat{\rho}(\alpha) \hat{\chi}_{[-1/2,1/2]^n}(\alpha)
= \hat{\rho}(\alpha) \prod\limits_{j:\ \alpha_j \neq 0} \frac{\sin(\alpha_j )}{\alpha_j}.
\end{align}
Here $\hat{u}(\xi):= \int\limits_{\R^n} u(x) e^{-i x\cdot \xi} dx$ denotes the Fourier transform for $u \in L^1(\R^n)$.
Now, after rescaling, we have that the vanishing of the mean value of $f_N$ is equivalent to 
\begin{align*}
\hat{\eta}(\alpha \sqrt{N})=0.
\end{align*}
It hence, remains to argue that for each fixed $\alpha \in \mathbb{S}^{n-1}$ there exists a sequence $(N_{k,\alpha})_{k \in \N} \subset \R$ such that $N_{k,\alpha} \rightarrow \infty$ as $k \rightarrow \infty$ with the property that
\begin{align*}
\hat{\eta}(\alpha \sqrt{N_{k,\alpha}})=0.
\end{align*}
With the explicit Fourier transform \eqref{eq:FT} in hand, since $\alpha \neq 0$, it follows that there exists such a sequence $(N_{k,\alpha})_{k \in \N} \subset \R$. 
\end{proof}

\begin{rmk}
The sequence $(N_{k,\alpha})_{k \in \N} \subset \R$ only depends on known information (i.e., the fixed functions $\rho$ and $\chi_{[-1/2,1/2]^n}$) and can hence be considered as \emph{known} in the measurement process.
\end{rmk}

With the previous observations on approximate solutions (Proposition \ref{prop:expansion_NtD}) and on the Neumann data (Lemma \ref{lem:cut-off}) in hand, for any $\alpha \in \mathbb{S}^{n-1}$ we define a normalized version of the function $\tilde{\bar{u}}_{N_{k,\alpha}}(x,x_{n+1})$ with data $f_{N_{k,\alpha}}$ as in Lemma \ref{lem:cut-off}:
\begin{align}
\label{eq:NtD_approx_sol}
\tilde{u}_{N_{k,\alpha}}(x,x_{n+1}):= \hat{c}_s^{-1} N_{k,\alpha}^{-2s} C_0^{-2s} \tilde{\bar{u}}_{N_{k,\alpha}}(x,x_{n+1}),
\end{align}
as well as the corresponding remainder from being an exact solution. More precisely, let $\tilde{w}_{N_{k,\alpha}}$ be a solution of
\eqref{eq:main_one1} with generalized Neumann data
\[f_{N_{k,\alpha}}(x) =c(x)\lim_{x_{n+1\to 0}}x_{n+1}^{1-2s}\p_{n+1}\tilde{u}_{N_{k,\alpha}}(x,x_{n+1})  \in C_{0,\diamond}^{\infty}(B_1(0)) \text{ for } N_{k,\alpha}>1. \]

Then we define the remainder of $\tilde{u}_{N_{k,\alpha}}$ from being a real solution as
\begin{align}
\label{eq:NtD_remainder}
r_{N_{k,\alpha}}(x,x_{n+1}):= \tilde{w}_{N_{k,\alpha}}(x,x_{n+1}) - \tilde{u}_{N_{k,\alpha}}(x,x_{n+1})
\end{align}

 As above, we also consider a cut-off function $\zeta \in C_c^{\infty}([0,\infty))$ with $\zeta(t) = 1$ for $t\in [0,1/2]$ and $\zeta(t) = 0$ for $t \geq 1$. We set $\zeta_{N_{k,\alpha}}(t) :=\zeta(\sqrt{N_{k,\alpha}} t) $. With slight abuse of notation, for abbreviation, in what follows below, we will drop the indeces in $N_{k,\alpha}$ and simply write $N$.

Since, compared to the leading order term $ v_0 $ in the Dirichlet case, in the approximate solutions, the leading order term $\tilde{v}_0$ in the function $\tilde{u}_{N_{k,\alpha}}$ is only rescaled, the argument from the Dirichlet case immediately implies that 
\begin{align*}
\nabla \tilde{u}_N(x,x_{n+1}) = (\hat{c}_s C_0^{2s} N^{2s})^{-1} e^{i N \alpha \cdot x} \begin{pmatrix} \alpha i N \eta_N(x) (\tilde{C}_0 N x_{n+1})^{s}K_s(\tilde{C}_0 N x_{n+1}) + R_N^1(x,x_{n+1})\\
-\eta_N(x) (C_0 N)^{s+1} x_{n+1}^s K_{s-1}(\tilde{C}_0 N x_{n+1}) + R_N^2(x, x_{n+1})
 \end{pmatrix},
\end{align*}
with $R_N^1, R_N^2$ error terms having
exactly the same asymptotic behaviour as in \eqref{eq:RN1RN2}. 

\subsubsection{Error estimates and proof of Proposition \ref{prop:NtD}}
Similarly as in the Dirichlet-to-Neumann case, we conclude the proof in the Neumann-to-Dirichlet setting by quantifying how well the function $\tilde{u}_N$ approximates the true solution $\tilde{w}_N$ by invoking the asymptotic behaviour of $\tilde{u}_N$.

\begin{proof}[Proof of Theorem \ref{prop:NtD}]
Building on the expansion from Proposition \ref{prop:expansion_NtD}, similarly as in the previous section, we proceed to deduce the leading order asymptotic expansion of the symbol of $\tilde{\Lambda}_{\gamma}$. The arguments are largely in parallel to the ones above, with some changes which will be highlighted below.
By definition of the Neumann-to-Dirichlet map, we have 
\[\begin{split}
\inp{\tilde{\Lambda}_\gamma f_N,f_N}_{H^s(\R^n),H^{-s}(\R^n)}&=\int_{\R^n}f_N(x)\tilde{w}_N(x,0)dx\\
&=\int_{\R^n}c(x)x_{n+1}^{1-2s}\p_{n+1}(\zeta_N\tilde{u}_N)\tilde{w}_N(x,0)dx.\\
\end{split} \]
Then, integrating by parts, 
\[ \begin{split}
	\inp{\tilde{\Lambda}_\gamma f_N,f_N}_{H^s(\R^n),H^{-s}(\R^n)}&=\int_{\R^{n+1}_+}x_{n+1}^{1-2s}\tilde{\gamma}\nabla (\zeta_N\tilde{u}_N)\cdot \nabla\tilde{w}_Ndxdx_{n+1}\\
	& \quad +\int_{\R^{n+1}_+}\nabla\cdot(\tilde{\gamma}\nabla(\zeta_N\tilde{u}_N))\tilde{w}_Ndxdx_{n+1}\\
	&=\int_{\R^{n+1}_+}x_{n+1}^{1-2s}\tilde{\gamma}\nabla (\zeta_N\tilde{u}_N)\cdot \nabla\tilde{u}_N dxdx_{n+1}\\
	& \quad +\int_{D_N\times \R_+}x_{n+1}^{1-2s}\tilde{\gamma}\nabla (\zeta_N\tilde{u}_N)\cdot \nabla r_Ndxdx_{n+1}\\
	&\qquad+\int_{D_N\times \R_+}\nabla\cdot(\tilde{\gamma}\nabla(\zeta_N\tilde{u}_N))\tilde{w}_Ndxdx_{n+1}\\
	&=:\widetilde{I}+\widetilde{II}+\widetilde{III},
\end{split}\]
where we recall that the approximate solution $ \tilde{u}_N $ is supported in $ D_N $ in the tangential direction. Then we see that, the term $ \widetilde{I} $ which eventually leads to boundary reconstruction  can be estimated in exactly the same manner as for the term $ I $ in the Step 3 of the proof of Theorem \ref{thm:bdry_reconstr_ext_nonloc_sol}, up to a different power in $ N $ due to the renormalization introduced in \eqref{eq:NtD_approx_sol}. Consequently, we have 
\[\lim_{N\to \infty}N^{2s+n/2}\widetilde{I}=\hat{c}_s^{-2}(c_1+c_2)(\bar{C}_\alpha(\gamma)(0))^{-2s} .\]
It remains to estimate the error terms. Let us start with $ \widetilde{II} $.  
By integrating by parts and using the equation for $r_N$,
\[
\widetilde{II}=\int_{D_N}\int_{0}^{\infty}(\nabla\cdot (x_{n+1}^{1-2s}\tilde{\gamma}\nabla \tilde{u}_N)) (\zeta_N\ol{\tilde{u}_N}) + \int_{\Gamma_1\cup\Gamma_2}x_{n+1}^{1-2s}\tilde{\gamma}\nabla r_N\cdot\nu \zeta_N\ol{\tilde{u}_N}=:A+B,
\]
where 
we have used $ \zeta_N(t)=0 $ if $ t>1/\sqrt{N} $ and we have set $ \Gamma_1:=\{|x|=1/\sqrt{N},\:0\leq x_{n+1} \} $ and $ \Gamma_2:=\{|x|\leq 1/\sqrt{N},\: x_{n+1}=0 \} $.
To estimate $ A $, recalling that by Proposition \ref{prop:expansion_NtD} (c),  $ \tilde{u}_N $ is an approximate solution in the sense that $ \nabla\cdot x_{n+1}^{1-2s}\tilde{\gamma}\nabla \tilde{u}_N=o(N^{1-k})z_{n+1}^{1-2s} $ as $ N\to\infty $, we obtain 
\[
\begin{split}
A&=N^{-1-\frac{n}{2}}\int_{|z|<1}\int_{0}^{\infty}o(N^{1-k})z_{n+1}^{1-2s}\left|\zeta\left(\frac{z_{n+1}}{\sqrt{N}}\right)\tilde{u}_N(z,z_{n+1})\right|dzdz_{n+1}\\
&\leq o(N^{-\frac{n}{2}-k})\int_{|z|<1}\int_{0}^{\infty}z_{n+1}^{1-2s}|\tilde{u}_N(z,z_{n+1})|dzdz_{n+1}.
\end{split}
 \]
Now note that $ \zeta(t)\leq 1 $ for any $ t\in[0,\infty) $, and that by Proposition \ref{prop:expansion_NtD} and the normalization \eqref{eq:NtD_approx_sol}, it holds that uniformly in $ z $,
\[\begin{split}
|\tilde{u}_N(z,z_{n+1})|&=N^{-2s} O(1)\text{ as }z_{n+1}\to 0,\\
|\tilde{u}_N(z,z_{n+1})|&=N^{-2s}O((C_0z_{n+1})^{s-\frac{1}{2}}e^{-C_0z_{n+1}}) \text{ as }z_{n+1}\to \infty,
\end{split} \]
where the implicit constant above is independent of $ N $. It then follows that $A=o(N^{-\frac{n}{2}-k-2s}) $.

Thus, when $ k\geq 0 $, we obtain
\[\lim_{N \rightarrow \infty}N^{n/2+2s}A=0. \]
We claim that $ B=0 $. Indeed, the integral in $ B $ over $ \Gamma_1 $ is zero since by construction $ \tilde{u}_N $ is supported on $ \{|x|<\frac{1}{\sqrt{N}} \} $, and the integral in $ B $ over $ \Gamma_2 $ is also zero as $ r_N $ has vanishing Neumann condition. In conclusion, we showed that $ \widetilde{II} $ is of order $ o(N^{-\frac{n}{2}-2s}) $.

Before we proceed to the estimate of $ \widetilde{III} $, we first establish a Poincar\'e type estimate. 
This follows from the fundamental theorem of calculus and the fact that the weight is a Muckenhoupt weight. More precisely, we have for a solution $u$ of \eqref{eq:main_one1} defined on $ \R^{n}\times [1,\frac{1}{\sqrt{N}}] $, for any $\epsilon>0$ and $ x_{n+1}\in [\epsilon,\frac{1}{\sqrt{N}}] $, by the fundamental theorem of calculus,
\begin{align*}
|u(x,x_{n+1})| ^2 \leq C\left( |u(x,\epsilon)|^2 + \left( \int\limits_{\epsilon}^{x_{n+1}} t^{\frac{2s-1}{2}} t^{\frac{1-2s}{2}} |\p_t u(x,t)|dt \right)^2 \right).
\end{align*}
We note that by the regularity of $\gamma$ for $\epsilon>0$ the solution $u$ is $C^2$ regular and, hence, the fundamental theorem is applicable.
By Hölder's inequality applied to the last term,  this gives rise to
\begin{align*}
|u(x,x_{n+1})| ^2 \leq C\left( |u(x,\epsilon)|^2 + x_{n+1}^{2s} \int\limits_{\epsilon}^{x_{n+1}}  t^{1-2s} |\p_t u(x,t)|^2dt  \right).
\end{align*}
Now multiplying  the above by $x_{n+1}^{1-2s}$ and integrating over $D_N \times [\epsilon,\frac{1}{\sqrt{N}}]$ then implies that
\begin{align*}
\int\limits_{D_N \times [\epsilon,\frac{1}{\sqrt{N}}]} x_{n+1}^{1-2s} |u(x,x_{n+1})| ^2 dx dx_{n+1}
& \leq CN^{s-1} \|u(x,\epsilon)\|_{L^2(D_N)}^2 \\
& \qquad  + CN^{-1} \int\limits_{D_N \times [\epsilon,\frac{1}{\sqrt{N}}]}  t^{1-2s} |\p_t u(x,t)|^2dt dx  .
\end{align*}
If moreover $ u\in\dot{H}^1(\R^{n+1}_+,x_{n+1}^{1-2s})   $, then by the trace estimate and Hölder's inequality, it holds that $ \|u(x,\epsilon)\|_{L^2(D_N)}^2\leq N^{-s}\|u(x,\epsilon)\|_{L^{\frac{2n}{n-2s}}(D_N)}^2  $. Therefore, letting $ \epsilon\to 0 $ and recalling the boundary trace estimate \eqref{eq:trace}, we are led to 

\begin{align}\label{eq:weightedL2}
\begin{split}
\int\limits_{D_N \times [0,\frac{1}{\sqrt{N}}]} x_{n+1}^{1-2s} |u(x,x_{n+1})| ^2 dx dx_{n+1}
&\leq CN^{-1}\|x_{n+1}^{1-2s}\nabla u\|_{L^2(\R^{n+1}_+)}^2 + CN^{-1}\|u(\cdot,0)\|_{L^{\frac{2n}{n-2s}}(\R^n)}^2\\
&\leq   CN^{-1}\|x_{n+1}^{1-2s}\nabla u\|_{L^2(\R^{n+1}_+)}^2 .	
\end{split}
\end{align}
With this in hand, to estimate $ \widetilde{III} $, we split it into two parts,
\[\begin{split}
\widetilde{III}&=\int_{D_N}\int_{0}^{\frac{1}{2\sqrt{N}}}\nabla\cdot (x_{n+1}^{1-2s}\tilde{\gamma}\nabla\tilde{u}_N) \tilde{w}_N
dxdx_{n+1}
+\int_{D_N}\int^{\frac{1}{\sqrt{N}}}_{\frac{1}{2\sqrt{N}}}\nabla\cdot\big( x_{n+1}^{1-2s}\tilde{\gamma}\nabla(\zeta_N\tilde{u}_N)\big) \tilde{w}_N\\
&=:\widetilde{III}_1+\widetilde{III}_2.
\end{split}
\]
For $ \widetilde{III}_1$, we use the fact that $ \tilde{u}_N $ is an approximate solution as stated in Proposition \ref{prop:expansion_NtD} so that 
\[\begin{split}
|\widetilde{III}_1|&= o(N^{-2s+2-k})\int_{D_N}\left|\int_{0}^{\frac{1}{2\sqrt{N}}}x_{n+1}^{1-2s}\tilde{w}_N \right|\\
&=o(N^{-2s+2-k})\left(\int_{D_N}\int_{0}^{\frac{1}{2\sqrt{N}}}x_{n+1}^{1-2s} \right)^{1/2}\left(\int_{D_N}\int_{0}^{\frac{1}{2\sqrt{N}}}x_{n+1}^{1-2s}|w(x,x_{n+1})|^2dxdx_{n+1} \right)^{1/2}\\
&=o(N^{-2s+2-k})(N^{-\frac{n}{2}+s-1})^{1/2} N^{-\frac{1}{2}}\|x_{n+1}^{\frac{1-2s}{2}}\nabla\tilde{w}_N\|_{L^2(\R^{n+1}_+)}\\
&=o(N^{-2s+2-k})(N^{-\frac{n}{2}+s-2})^{1/2}\|f_N\|_{\dot{H}^{-s}(\R^n)},
\end{split}\]
where in the second last line we applied the weighted $ L^2 $ estimate \eqref{eq:weightedL2}, and for the last line the energy estimate.
We compute that 
\[\|f_N\|_{\dot{H}^{-s}(\R^n)}\leq C\|xf_N\|_{H^1(\R^n)}+\|f_N\|_{L^2(B_1(0))}\leq CN^{-\frac{n}{4}+\frac{1}{2}}. \]
Therefore, we have
\[|\widetilde{III}_1|=o(N^{-\frac{n}{2}-2s+\frac{s-1}{2}+2-k}), \]
which implies that when $ k\geq 2 $, $ \lim_{N\to\infty}N^{n/2+2s}\widetilde{III}_1=0 $. We now turn to $ \widetilde{III}_2 $. By integration by parts, 
\[\begin{split}
\widetilde{III}_2&=\int_{D_N}(\sqrt{N}/2)^{2s-1}c(x)\p_{n+1}\tilde{u}_N(x,\frac{1}{2\sqrt{N}})\tilde{w}_N(x,\frac{1}{2\sqrt{N}})dx-\int_{D_N}\int_{\frac{1}{2\sqrt{N}}}^{\frac{1}{\sqrt{N}}}x_{n+1}^{1-2s}\tilde{\gamma}\nabla(\zeta_N\tilde{u}_N)\cdot\nabla\tilde{w}_N\\
&=:\widetilde{III}_3+\widetilde{III}_4.
\end{split} \]
For $ \widetilde{III}_4 $, we make use of the exponential decay of $ \nabla \tilde{u}_N $ and $ \tilde{u}_N $ to conclude that
\[\begin{split}
|\widetilde{III}_4|&\leq\|x_{n+1}^{1-2s}\nabla\tilde{w}_N\|_{L^2(\R^{n+1}_+)}\left( \int_{D_N}\int_{\frac{1}{2\sqrt{N}}}^{\infty}|\nabla\zeta_N\tilde{u}_N+\zeta_N\nabla\tilde{u}_N|^2 \right)^{1/2}\\
&=O(N^{-\frac{n}{4}+\frac{1}{2}})O(e^{-\frac{C_0\sqrt{N}}{4}})=o(N^{-n/2-2s}).
\end{split} \]
Similarly, for the boundary term $ \widetilde{III}_3 $ we check that from \eqref{eq:approxsolgradient}, \eqref{eq:RN1RN2} and \eqref{eq:NtD_approx_sol} that as $ N\to \infty $ uniformly for $ x\in B_1(0) $ we have
\[\begin{split}
|\p_{n+1}\tilde{u}_N(x,\frac{1}{2\sqrt{N}})|&\leq N^{-2s}\left( C_0N^{\frac{s}{2}+1}K_{s-1}(C_0\sqrt{N}/2)+R_N^2(x,\frac{1}{2\sqrt{N}}) \right)=O(e^{-\frac{C_0\sqrt{N}}{4}}),
\end{split}
 \]
 where we have used the exponential decay of the Bessel function $ K_{s-1} $ and the reminder term $ R_N^2 $. With the above and Cauchy-Schwarz, we see that 
\[\begin{split}
|\widetilde{III}_3|&\leq CN^{s-1/2}\left(\int_{D_N}|\p_{n+1}\tilde{u}_N(x,\frac{1}{2\sqrt{N}})|^2dx \right)^{1/2}\left(\int_{D_N}|\tilde{w}_N(x,\frac{1}{2\sqrt{N}})|^2 dx\right)^{1/2}\\
&=O(e^{-\frac{C_0\sqrt{N}}{5}})\left(\int_{D_N}\int_{\frac{1}{2\sqrt{N}}}^{\frac{1}{\sqrt{N}}} |\nabla\tilde{w}_N(x,x_{n+1})|^2+|\tilde{w}_N(x,x_{n+1})|^2dxdx_{n+1}\right)^{1/2}\\
&=O(e^{-\frac{C_0\sqrt{N}}{5}})\left(2\sqrt{N}\int_{D_N}\int_{\frac{1}{2\sqrt{N}}}^{\frac{1}{\sqrt{N}}} x_{n+1}^{1-2s}\left(|\nabla\tilde{w}_N(x,x_{n+1})|^2+|\tilde{w}_N(x,x_{n+1})|^2\right)dxdx_{n+1}\right)^{1/2}\\
&=O(e^{-\frac{C_0\sqrt{N}}{5}})\left(2\sqrt{N}(1+CN^{-1})\|x_{n+1}^{\frac{1-2s}{2}}\nabla\tilde{w}_N\|^2_{L^2(\R^{n+1}_+)} \right)^{1/2}=O(e^{-\frac{C_0\sqrt{N}}{6}})
\end{split} \] 
where in the second line above we also used the usual trace estimate for $ \tilde{w}_N(\cdot,\frac{1}{2\sqrt{N}}) $, and for the second last line we used again the estimate \eqref{eq:weightedL2}. Combining all the estimates obtained above, we infer that $ \lim_{N\to\infty}N^{n/2+2s}\widetilde{III}=0 $, which concludes the proof.
\end{proof}

\begin{rmk}
We comment on the $ C^4 $ regularity assumptions for the metric $ \gamma$. In the construction of the approximate solutions, this assumption allows us to expand $ \gamma $ (with Taylor expansion) around the point $ x_0 $ where we perform boundary reconstruction, which eventually leads to the hierarchy of ODEs in the construction of the approximate solutions, see Step 1 of proof of Proposition \ref{prop:DtN_approx}.
For the error estimate for the approximate solution in both Dirichlet-to-Neumann and Neumann-to-Dirichlet cases, the assumption that  $ \gamma\in C^{2k} $ for some $ k\geq 1 $ allows us to deduce that the error terms are of lower order, see the estimate of $ A $ in the proof of Theorem \ref{thm:Eucl} and of $ \widetilde{III}_1 $ in the proof of Theorem \ref{prop:NtD}. 
While we believe that these conditions can be reduced, we do not pursue this here.
\end{rmk}

\subsection{The geometric source-to-solution problem}
\label{sec:mfd2}

Last but not least, we consider the geometric source-to-solution problem which, after extension, in local coordinates takes a very similar form to the problem discussed in the previous section. More precisely, the proofs of the geometric versions formulated in Theorems \ref{thm:bdry_reconstr_ext_nonloc_source} and \ref{thm:bdry_reconstr_ext} follow as the ones in Section \ref{sec:mfd} by adopting an extension perspective.

Let us first recall the set-up. As in the Dirichlet case, it is possible to adopt an extension perspective on the source-to-solution problem. More precisely, considering \eqref{eq:Neumann_Dirichlet_main}, then we have that $\bar{L}_{s,O}(f) = \tilde{u}^f|_{O \times \{0\}}$. As above, also in the manifold setting a compatibility condition has to be satisfied by the data to ensure unique solvability of the problem: in the geometric context it turns into $\int\limits_{M} \frac{1}{\sqrt{\det(g)}} f dV_g = 0$. Since we will construct our test data $f_N$ to be supported in a sufficiently small patch around $x_0 \in O$,
this can be achieved by a variant of Lemma \ref{lem:cut-off}. We present the central steps in what follows next.

\begin{proof}[Proof of Theorems \ref{thm:bdry_reconstr_ext_nonloc_source} and \ref{thm:bdry_reconstr_ext}]
As in the solution-to-source setting the proofs of Theorems \ref{thm:bdry_reconstr_ext_nonloc_source} and \ref{thm:bdry_reconstr_ext} will be obtained simultaneously. As in the solution-to-source setting,  the proof consists of noting that the testing boundary data and the approximate solution are localized, thus computing the duality paring $ \langle f_N, \frac{1}{\sqrt{\det(g)}} \bar{L}_{s,O}(f_N) \rangle_{H^{s}(M,dV_g), H^{-s}(M, dV_g)} $ in local coordinates reduces the problem to the Euclidean case considered above.\\

\emph{Step 1: Construction of approximate solutions and solvability.} We begin by outlining the main differences  compared to Proposition \ref{prop:expansion_NtD} in the construction of approximate solutions and the associated solvability questions. As in Proposition \ref{prop:expansion_NtD}, we consider highly oscillatory and localized data $ (f_N)_{N \in \N} $. Hence, we may work in a single coordinate patch. Assuming, without loss of generality, that $x_0 = 0$, after passing to local coordinates, the compatibility condition for solvability for the equation 
\begin{align}
\begin{split}
\label{eq:geo_NtD}
(\p_{n+1} x_{n+1}^{1-2s} \p_{n+1} + x_{n+1}^{1-2s} \D_g) \tilde{u}^f & = 0 \mbox{ on } M \times \R_+,\\
\sqrt{\det(g)} \lim\limits_{x_{n+1} \rightarrow 0} x_{n+1}^{1-2s} \p_{n+1} \tilde{u}^f & = f \mbox{ on } M,
\end{split}
\end{align}
 turns into
\begin{align*}
\int\limits_{B_r(0)} f_N(x) dx = 0,
\end{align*}
for some suitable radius $r>0$. Without loss of generality, in what follows, we will assume that $r=1$.
This brings us back to the compatibility condition from Lemma \ref{lem:cut-off} which ensures the existence of a (known) sequence $(N_{k,\alpha})_{k \in \N}$ for any $\alpha \in \mathbb{S}^{n-1}$.
In particular, an exact solution $\tilde{w}_{N_{k,\alpha}}$ to \eqref{eq:geo_NtD} can be constructed for such generalized Neumann data $f_{N_{k,\alpha}}$. As before, we will drop the indeces in $ N_{k,\alpha} $ below.

Now, let $ c(x)=\sqrt{\det g(x)} $ and we define (in local coordinates)
\[\tilde{\gamma}(x,x_{n+1})=\sqrt{\det g(x)}\begin{pmatrix}
	g^{-1}(x) & 0\\
	0 & 1
\end{pmatrix}, \]
and then by Proposition \ref{prop:expansion_NtD}, in this coordinate patch, we obtain the approximate solutions $\tilde{u}_N$ with Neumann boundary data being $ f_N $. 

\emph{Step 2: The generalized Alessandrini identity.} With existence of approximate solutions at our disposal, we turn towards the Alessandrini identity in the geometric context as this is the basis for the reconstruction result. First note that, comparing the extension result in \eqref{eq:geo_NtD} and the usual setting in the local case \eqref{eq:main_one1}, we have that 
\[\bar{L}_{s,O}\left( f_N\right)=c_s^{-1}\tilde{w}_N(\cdot,0).  \]
Thus,   as in the proof of Theorems  \ref{thm:bdry_reconstr_ext_nonloc_sol} and \ref{thm:bdry_reconstr_ext_NtD}, we observe that by integration by parts
 \begin{align*}
&\langle f_N, \frac{1}{\sqrt{\det(g)}} \bar{L}_{s,O}(f_N) \rangle_{H^{s}(M,dV_g), H^{-s}(M, dV_g)}\\
&=c_s^{-1}\int\limits_{M}\tilde{w}_N(x)c(x)x_{n+1}^{1-2s}\partial_{n+1}(\tilde{u}_N\zeta_N) \frac{1}{\sqrt{\det(g)}} dV_g  \\
&=c_s^{-1} \int_{B_1(0)\times\R_+} x_{n+1}^{1-2s}\nabla w_N\cdot\tilde{\gamma}\nabla(\tilde{u}_N\zeta_N)dxdx_{n+1}+c_s^{-1}\int_{B_1(0)\times\R_+}\nabla\cdot(\tilde{\gamma}\nabla(\zeta_N\tilde{u}_N))\tilde{w}_Ndxdx_{n+1} .
\end{align*}
Now arguing as in the proof of Theorem \ref{thm:bdry_reconstr_ext_nonloc_sol}, working in a single coordinate patch,  exactly the same proof as in the Euclidean case (that is, the proof of Theorem \ref{prop:NtD})  leads to the correct error bound and boundary reconstruction.
\end{proof}
\begin{proof}[Proof of Corollary \ref{cor:metric}]
We note that also in the solution-to-source measurement setting, all quantities in 
	\begin{align*}
		\langle f_N, \frac{1}{\sqrt{\det(g)}} \bar{L}_{s,O}(f_N) \rangle_{H^{s}(M,dV_g), H^{-s}(M, dV_g)}
	\end{align*} 
	are available from the measurement data.
\end{proof}

\section{Proof of the Applications from Theorems \ref{thm:FGKU_nonlocal}, \ref{thm:FGKU_nonlocal1a}  and \ref{thm:local_nonlocal1}}
\label{sec:consequences}

In this last section, we discuss applications of the above theorems. We begin by discussing the proof of Theorem \ref{thm:FGKU_nonlocal}. 

\begin{proof}[Proof of Theorem \ref{thm:FGKU_nonlocal}]
The proof of Theorem \ref{thm:FGKU21} follows from an application of Theorem \ref{thm:bdry_reconstr_ext_nonloc_source} or Theorem \ref{thm:bdry_reconstr_ext} and the result of \cite{FGKU21}. Indeed, by Theorem \ref{thm:bdry_reconstr_ext_nonloc_source} and, equivalently, Theorem \ref{thm:bdry_reconstr_ext} we have that if $ \bar{L}_{s,O}^1 =  \bar{L}_{s,O}^2$ are known, then the full metrics $g_1, g_2$ can be reconstruced in $O$ and it holds $g_1|_{O} = g_2|_{O}$. With this, the assumption of the main theorem in \cite{FGKU21} are satisfied and the remainder of Theorem \ref{thm:FGKU_nonlocal} then follows from Theorem 1 in \cite{FGKU21}.
\end{proof}

The proof of Theorem \ref{thm:FGKU_nonlocal1a} follows analogously. We, hence, omit its detailed discussion and turn to the proof of Theorem  \ref{thm:FGKU_nonlocal1a}. 

\begin{proof}[Proof of Theorem \ref{thm:local_nonlocal1}]
Also here, by Theorems \ref{thm:bdry_reconstr_ext_nonloc_source} and Theorem \ref{thm:bdry_reconstr_ext}, we have that the knowledge of $\bar{L}_{s,O}^1 =  \bar{L}_{s,O}^2$ allows us to reconstruct $g_1, g_2$ in $O$. This implies that again $g_1|_{O} = g_2|_{O}$. With the assumption of the main theorem in \cite{R23} are satisfied and the remainder of Theorem \ref{thm:local_nonlocal1} then follows from \cite{R23}.
\end{proof}

\section*{Acknowledgements}
This is an extension of the first author's master's thesis written at the University of Bonn. A.R.~gratefully acknowledges funding from the Deutsche Forschungsgemeinschaft (DFG, German Research Foundation) through the Hausdorff Center for Mathematics under Germany's Excellence Strategy - EXC-2047/1 - 390685813.

\bibliographystyle{alpha}
\bibliography{citations}

\end{document}